\documentclass[11pt]{amsproc}

\usepackage{amscd, amsmath, amsthm, amssymb, mathtools, mathrsfs, stmaryrd}
\usepackage{setspace}
\usepackage{comment}
\usepackage{bm, bbm}
\allowdisplaybreaks
\usepackage[all]{xy}
\usepackage{fullpage}
\usepackage{pdflscape,lipsum,etoolbox}
\usepackage[pagebackref]{hyperref}

\usepackage{tikz}
\usetikzlibrary{intersections, calc, arrows.meta}
\usepackage{here}
\usepackage{time}

\usepackage{xcolor}

\usepackage{amsrefs}

\usepackage[capitalize,nameinlink,noabbrev,nosort]{cleveref}

\hypersetup{
	colorlinks=true,       
	linkcolor=blue,          
	citecolor=blue,        
	filecolor=blue,      
	urlcolor=blue,           
}

\makeatletter
\@namedef{subjclassname@2020}{\textup{2020} Mathematics Subject Classification}
\makeatother

\newtheorem{theoremcounter}{Theorem Counter}[section]

\theoremstyle{definition}
\newtheorem{definition}[theoremcounter]{Definition}

\newtheorem{remark}[theoremcounter]{Remark}
\theoremstyle{plain}
\newtheorem{lemma}[theoremcounter]{Lemma}
\newtheorem{proposition}[theoremcounter]{Proposition}
\newtheorem{corollary}[theoremcounter]{Corollary}

\newtheorem{theorem}[theoremcounter]{Theorem}

\numberwithin{equation}{section}

\newcommand{\Z}{\mathbb{Z}}
\newcommand{\Q}{\mathbb{Q}}
\newcommand{\R}{\mathbb{R}}
\newcommand{\C}{\mathbb{C}}

\def\1{\mathbbm{1}}

\def\GL{\mathrm{GL}}
\def\O{\mathrm{O}}
\DeclareMathOperator{\Gal}{Gal}

\DeclareMathOperator{\rank}{rank}

\newcommand{\fS}{\mathfrak{S}}

\newcommand{\cO}{\mathcal{O}}
\newcommand{\cP}{\mathcal{P}}

\newcommand{\cR}{\mathcal{R}}

\begin{document}

\title[]{Root lattices over totally real fields}

\author{Ryotaro Sakamoto}
\address{Department of Mathematics,  University of Tsukuba,  1-1-1 Tennodai,  Tsukuba,  Ibaraki 305-8571,  Japan}
\email{rsakamoto@math.tsukuba.ac.jp}

\author[]{Miyu Suzuki}
\address{Department of Mathematics, Kyoto University, Kitashirakawa Oiwake-cho, Sakyo-ku,  Kyoto 606-8502, Japan}
\email{suzuki.miyu.4c@kyoto-u.ac.jp}

\author{Hiroyoshi Tamori}
\address{Department of Mathematical Sciences, 
Shibaura Institute of Technology, 307 Fukasaku, Minuma-ku, Saitama City, Saitama, 337-8570, Japan}
\email{tamori@shibaura-it.ac.jp}

\subjclass[2020]{}


\maketitle
\renewcommand{\thefootnote}{}
\renewcommand{\thefootnote}{\arabic{footnote}}

\renewcommand{\leftmargini}{25pt}

\begin{abstract}
A root lattice is a finite rank $\Z$-lattice generated by elements $x$ satisfying $x\cdot x=2$.
It is well-known that the root lattices have an $ADE$ classification and they play a prominent role in the study of even unimodular lattices.
The notion of root lattices can be naturally generalized to lattices over the ring of integers $\cO$ of a totally real field $K$.
In the case where $K$ is a real quadratic field, such lattices were classified  by Mimura in 1979, and 
this classification has been used by several researchers in the study of even unimodular $\mathcal{O}$-lattices.
In this paper, we extend this classification to arbitrary totally real fields. 
The irreducible root lattices of rank greater than $2$ are indexed by finite Coxeter systems.
All the rank $2$ root lattices are realized as orders in quadratic extensions of $K$ and their classification requires some technique from algebraic number theory.
\end{abstract}

\tableofcontents

\section{Introduction}\label{sec:introduction}

A \emph{root lattice} is a finitely generated free $\Z$-module $L$ generated by the set of roots $\Phi(L)\coloneqq\{x\in L \mid x\cdot x=2\}$.
Here,  $(x,  y)\mapsto x\cdot y$ is a positive definite symmetric bilinear form on the real vector space $V\coloneqq L\otimes_\Z\R$ 
of positive dimension 
satisfying $x\cdot y\in\Z$ for all $x,  y\in L$.
Root lattices play a prominent role in the classification of $24$-dimensional even unimodular lattices,  namely Niemeier lattices.
See \cite[Chapter 18]{CS} and \cite[Chapter 3]{Ebeling13}.

It is known that every root lattice admits a unique direct sum decomposition into irreducible root lattices, which are classified by the $ADE$-types. 
To be precise,  if $L$ is an irreducible root lattice,  then $\Phi(L)$ is an irreducible simply-laced root system and is classified using the Dynkin diagrams of type $A_n$ ($n\geq 1$),  $D_n$ ($n\geq4$),  and $E_n$ ($n=6,  7,  8$).  
For the details,  we refer the readers to \cite[Chapter 1]{Ebeling13}.

\textbf{In the present paper,  we generalize this classification to the root lattices over totally real fields.}
Let $K$ be a totally real number field, that is, $K$ is an algebraic (possibly infinite) extension of $\mathbb{Q}$ and all embeddings $K \hookrightarrow \mathbb{C}$ have image in $\mathbb{R}$. 
Let  $\cO \coloneqq \cO_K$ denote the ring of integers of $K$, that is, the set of all elements of $K$ that are integral over $\mathbb{Z}$. 
Suppose that $L$ is a 
nonzero 
finitely generated free $\cO$-module with a non-degenerate symmetric bilinear form 
    \[
    L\times L\to\cO; \quad (x,y) \mapsto x\cdot y.
    \] 
We say that $L$ is an \emph{$\cO$-lattice} if the bilinear form is totally positive definite,  \emph{i.e.},  for all real embedding $\sigma\,\colon K \hookrightarrow \R$,  we have
    \[
    \sigma(x\cdot x)>0,  \quad x\in L\setminus\{0\}.
    \]   
An $\cO$-lattice $L$ is called a root lattice if it is generated by $\Phi(L)\coloneqq\{x\in L \mid x\cdot x=2\}$. 

In the case where $K$ is a real quadratic field, root lattices were classified by Mimura in \cite{Mimura79} using an ad-hoc method. 
Based on that classification,  even unimodular lattices over real quadratic fields are studied by many researchers,  see for example \cite{Takada85},  \cite{CH87} \cite{HH89},  \cite{Hsia89} and \cite{Scharlau94}. 
However, to our knowledge, the authors could not find any references on root lattices over general totally real fields.
Scharlau \cite[\S 2.6]{Scharlau94} mentioned that several finite Coxeter groups occur as Weyl groups of root lattices over real quadratic fields,  but did not consider general totally real fields,  saying that $\langle\hspace{-.2em}\langle$Since so little is known in general,  the restriction to quadratic fields seems to be quite natural at present and is also adopted in our work.$\rangle\hspace{-.2em}\rangle$

Now we state our main results. 
First, we fix an embedding of $K$ into $\mathbb{R}$, and regard
$K$ as a subfield of $\mathbb{R}$. For simplicity,  we also assume that $[K:\Q]<\infty$.
Let $L$ be a root lattice over $\cO$. We then observe that $\Phi(L)$ is a finite set and has a natural structure of a root system.
Then it admits a fundamental system of roots $\Delta \coloneqq \Delta(L) \subset \Phi(L)$, which is 
\begin{itemize}
\item a basis of the $K$-vector space $L\otimes_\cO K$ and
\item each $\alpha\in\Phi(L)$ is written in the form $\alpha=\sum_{\beta\in\Delta}c_\beta\beta$ with $c_\beta\in\cO_{\geq1}\sqcup\{0\}$ for all $\beta\in\Delta$ or $-c_\beta\in\cO_{\geq1}\sqcup\{0\}$ for all $\beta\in\Delta$. Here $\cO_{\geq 1} \coloneqq  \cO \cap \mathbb{R}_{\geq 1}$. 
\end{itemize}

We prove in \cref{fund_negative} that for each $\alpha,  \beta\in\Delta$,  there exists a unique positive integer $m \coloneqq m(\alpha,  \beta)$ such that 
    \[
    \alpha\cdot \beta=-2\cos(\pi/m).
    \]
Hence we can associate the Coxeter-Dynkin diagram to $L$,  whose vertices are the roots in $\Delta$ and the edges are labeled by $m(\alpha,  \beta)$. 
It turns out that this diagram is the one corresponding to a finite Coxeter system (see \cref{injection}).  
The possible diagrams are recalled in \cref{Coxeter_classification}.
They are one of the following types: $A_n$ ($n\geq1$),  $B_n$ ($n\geq2$),  $D_n$ ($n\geq4$),  $E_n$ ($n=6,7,8$),  $F_4$,  $H_n$ ($n=3,4$),  and $I_2(m)$ ($m\geq5$).

The problem is the existence of an irreducible root lattice  over $\cO$ of type $X_n$,  where $X_n$ is one of the labels listed above.
We note that the root lattice of type $X_n$ has $\mathcal{O}$-rank $n$. 
For each $X_n$,  we set 
    \[
    c(X_n) = 
        \begin{cases}
        1 & \text{if $X_n=A_n$ ($n\geq1$),  $D_n$ ($n\geq 4$),  or $E_n$ ($n=6,  7,  8$)},  \\
        \sqrt{2} & \text{if $X_n=B_n$ ($n\geq2$) or $F_4$},  \\
        (1+\sqrt{5})/2 & \text{if $X_n=H_n$ ($n=3,  4$)}, \\
        2\cos(\pi/m) & \text{if $X_n=I_2(m)$ ($m\geq5$).}
        \end{cases}
    \]
When the $\cO$-rank is greater than $2$,  we prove the next theorem.

\begin{theorem}\label{main-intro1}
Let $K$ be a totally real number field and $\cO=\cO_K$ its ring of integers.
Suppose that $n$ is an integer greater than $2$.
Then,  there exists an irreducible root lattice over $\cO$ of type $X_n$ if and only if $c(X_n)\in K$.
\end{theorem}

Note that \cref{main-intro1} holds for all totally real fields $K$,  without assuming $[K:\Q]<\infty$.
For more details,  see \cref{sec:scalar}.

When the $\cO$-rank is $2$,  the situation is drastically different. 
For any positive integer $n$,  let $\zeta_n\in\bar{K}$ denote a primitive $n$-th root of unity.
We define a directed graph with the set of vertices 
    \[
    Q_K= \{n > 1 \mid \zeta_{2n}+\zeta_{2n}^{-1}\in K\}
    \]
and the vertices $x, y \in Q_K$ are connected by an edge directed from $x$ to $y$ if $x$ divides $y$ and $y/x$ is a prime number. 
By an abuse of notation,  we denote this graph by $Q_K$.
We also define the partial order on $Q_K$ such that $x \preceq y$ if there is a path from $x$ to $y$.

Let $\pi_0(Q_K)$ denote the set of connected components of the graph $Q_K$ and define the subsets $\cP_K$ and $\cR_K$ of the power set of $Q_K$ as follows:
    \begin{align*}
    \cP_K &\coloneqq  \{ \{q\}  \mid \text{$q \in Q_K$ is a prime power} \},  \\
    \cR_K &\coloneqq  \{ C \in \pi_0(Q_K) \mid \text{$C$ contains an integer that is not a prime power} \}. 
    \end{align*}
Let   $\mu_\infty\subset\bar{K}^\times$ denote the group of all roots of unity.  
For each $C \in\cP_K\cup\cR_K$,  let $\mu(C)$ be the subgroup of $\mu_\infty$ generated by $\{ \zeta_{2n} \mid n \in C  \}$. 
The next theorem classifies all the rank $2$ root lattices.

\begin{theorem}\label{main-intro2}
For any root lattice $L$  over $\cO$ of rank $2$,  let $\mu(L)$ denote the subgroup of $\mu_\infty$ generated by $\{\zeta \in \mu_{\infty} \mid \zeta+\zeta^{-1} = \alpha\cdot\beta  \text{ for some $\alpha,  \beta\in\Phi(L)$} \}$.  
Then we have $\mu(L)=\mu(C)$ for some $C\in\cP_K\cup\cR_K$ and the map 
    \[
    \{ \text{root lattices over $\cO$ of rank $2$} \}_{/\simeq}
    \to \{\mu(C) \mid C\in\cP_K\cup\cR_K\};  \quad L \mapsto \mu(L).
    \]
is a bijection.

In particular when $[K:\Q]<\infty$,  there exists an irreducible root lattice over $\cO$ of type $I_2(m)$ ($m\geq 3$) if and only if $K$ satisfies the following conditions (i) and (ii):
\begin{itemize}
\item[(i)] We have $c(I_2(m))=2\cos(\pi/m)\in K$ \emph{i.e.,} $m\in Q_K$. 
\item[(ii)] The integer $m$ is a prime power or a maximal element of $Q_K$.
\end{itemize}
\end{theorem} 
For more precise statement,  see \cref{summary>2},  \cref{thm:rank_2_root_lattice},  and \cref{thm:rank_2_root_lattice2}.
The case of $K=\Q$ of our result recovers the $ADE$ classification of the ordinary root lattices.
When $[K:\Q]=2$,   our classification concincides with the one obtained by Mimura in \cite{Mimura79}.

This article is organized as follows. 
In \cref{sec:preliminary},  we introduce necessary notations for root lattices and recall the classification of finite Coxeter systems.
In \cref{sec:root system},  we show that the set of roots $\Phi(L)$ of a root lattice $L$ has a fundamental system of roots assuming $|\Phi(L)|<\infty$.
We associate Coxeter-Dynkin diagrams with irreducible root lattices and observe that they correspond to finite Coxeter systems.
\cref{sec:scalar} is devoted to proving the following: for an irreducible root lattice of rank greater than $2$, the set of roots does not change under extension of scalars. 
In particular,  we see that an irreducible root lattice $L$ with $\rank(L)\geq3$ satisfies $|\Phi(L)|<\infty$, and we deduce \cref{main-intro1} from this.
The rank $2$ root lattices are studied in \cref{sec:rank2}.
First we observe a rank $2$ root lattice is isomorphic to an order $\cO[\zeta]$ in a quadratic extension $K(\zeta)$ of $K$.
The group of roots of unity in $\mathcal{O}[\zeta]$ is a complete invariant of its isomorphism class as a root lattice. 
We are thus reduced to studying possible orders of the form $\mathcal{O}[\zeta]$, which can be analyzed by examining the ramification at each prime.

\subsection*{Acknowledgment}

The authors would like to thank Tamotsu Ikeda and Hiroyuki Ochiai for valuable comments and kindly answering questions. 
R.S.  was supported by JSPS KAKENHI Grant Number JP24K16886. 
M.S.  was supported by JSPS KAKENHI Grant Number  JP22K13891. 
H.T.  was supported by JSPS KAKENHI Grant Number  JP23K12947.

\section{Preliminaries}\label{sec:preliminary}

\subsection{Definitions of root lattices}\label{sec:definition}

Let $K$ be a totally real number field and $\cO \coloneqq \cO_K$ its ring of integers.
We fix an embedding $K\hookrightarrow\R$ once for all and consider the total order on $K$ induced from this embedding. 
Suppose that $V$ is a $K$-vector space of dimension $n\in\Z_{>0}$ with a symmetric bilinear form
    \[
    V\times V \to K; \quad (x,  y) \mapsto x\cdot y.
    \]
Throughout this article,  we assume that this symmetric bilinear form is \emph{totally positive definite},  that is,  for any real embedding $\sigma \colon K \hookrightarrow \R$ and any non-zero vector $x\in V$,  we have $\sigma(x\cdot x)>0$.

\begin{definition}
An \emph{$\cO$-lattice} in $V$ is a finitely generated $\cO$-submodule $L\subset V$ with the property that there exists a basis $v_1,  v_2,  \ldots,  v_n$ of $V$ over $K$ such that $L=\cO v_1+\cO v_2+\cdots+\cO v_n$.
\end{definition}

We say that an $\cO$-lattice $L$ in $V$ is \emph{reducible} if there exist an orthogonal decomposition $V=V_1\oplus V_2$ with $V_1 \neq 0$ and $ V_2\neq0$, and $\cO$-lattices $L_i$ in $V_i$ for $i=1,  2$ such that $L=L_1\oplus L_2$. 
A lattice is called \emph{irreducible} if it is not reducible.

\begin{definition}
\begin{itemize}\setlength{\itemsep}{5pt}
\item[(1)] An $\cO$-lattice $L \subset V$  is called \emph{integral} if $x\cdot y\in\cO$ for all $x,  y\in L$.
\item[(2)] Set $\Phi(L) \coloneqq \{x\in L \mid x\cdot x=2\}$.
An element of $\Phi(L)$ is called a \emph{root} of $L$.
\item[(3)] An $\cO$-lattice $L\subset V$ is called a \emph{root lattice} over $\cO$ if it is integral and generated by $\Phi(L)$.
\end{itemize}
\end{definition}

From now until the end of this article,  let $L\subset V$ be a root lattice over $\cO$ and set $\Phi\coloneqq\Phi(L)$.
Let $K_{\geq1}$ denote the elements of $K$ not smaller than $1$ with respect to the order induced from the fixed embedding $K\hookrightarrow\R$.
We set $\cO_{\geq1}=K_{\geq1}\cap\cO$.

\begin{definition}\label{def_fund}
A subset $\Delta\coloneqq\Delta(L)$ of $\Phi$ is called a \emph{fundamental system} of roots if it satisfies the following conditions (i) and (ii):
\begin{itemize}\setlength{\itemsep}{5pt}
\item[(i)] The set $\Delta$ is a basis of the vector space $V$.
\item[(ii)] Each root $\alpha\in\Phi$ can be written as $\alpha=\sum_{\beta\in\Delta}c_\beta\beta$ with $c_\beta\in\cO_{\geq1}\sqcup\{0\}$ for all $\beta\in\Delta$ or $-c_\beta\in\cO_{\geq1}\sqcup\{0\}$ for all $\beta\in\Delta$.
\end{itemize}
\end{definition}

A priori we do not know the existence of a fundamental system of roots.
This is proved in \cref{sec:root system} when $[K:\Q]<\infty$ and in \cref{sec:scalar} when $\rank(L)>2$ for general $K$.
Note that a fundamental system of roots may not exist when $\rank(L)=2$ and $[K:\Q]=\infty$.

The following lemma is a well-known result due to Kronecker.
For the reader's convenience, we include a proof. 

\begin{lemma}[Kronecker's theorem]\label{Kronecker}
Let $\alpha$ be a non-zero algebraic integer.
\begin{itemize}\setlength{\itemsep}{5pt}
\item[(1)] If all conjugates of $\alpha$ over $\mathbb{Q}$ have absolute value at most $1$, then $\alpha$ is a root of unity. 
\item[(2)] If all conjugates of $\alpha$ over $\Q$ are real numbers with absolute values at most $2$,  then $\alpha=2\cos(r\pi)$ for some $r\in\Q$.
\item[(3)] If all conjugates of $\alpha$ over $\Q$ are real numbers with absolute values at most $\sqrt{2}$,  then $\alpha \in \{0,  \pm 1, \pm \sqrt{2}\}$.
\end{itemize}
\end{lemma}

\begin{proof}
Let $\bar{\Z}$ denote the ring of all algebraic integers.
Let $\{\alpha_1,  \alpha_2,  \ldots,  \alpha_r\}$ be the set of all conjugates of $\alpha$ over $\Q$.

\item[(1)]  We present a proof for completeness, adapted from \cite[Theorem 1.5.9]{Bombieri09}. 
For any  integer $k \in \{0,1,\ldots,  r\}$,  let 
    \[
    s_k(t_1,  t_2,  \ldots,  t_r)
    \coloneqq \sum_{1\leq i_1<i_2<\cdots<i_k\leq r}t_{i_1}t_{i_2}\cdots t_{i_k}
    \]
be the elementary symmetric polynomial of degree $k$.
Then for any positive integer $n$,  we have $s_k(\alpha_1^n,  \alpha_2^n,  \ldots,  \alpha_r^n)\in\bar{\Z}\cap\Q=\Z$. 
On the other hand,  since we have $|\alpha_j|\leq1$ for each $j$, it follows that 
    \[
    \sum_{k=0}^r |s_k(\alpha_1^n,  \alpha_2^n,  \ldots,  \alpha_r^n)| 
    \leq \sum_{k=0}^r s_k(1,\ldots, 1) = \sum_{k=0}^r 
        \begin{pmatrix}
        r \\
        k
        \end{pmatrix}
    = 2^r.
    \]
In particular,  $s_k(\alpha_1^n,  \alpha_2^n,  \ldots,  \alpha_r^n)$ is an integer between $-2^r$ and $2^r$ for all $k=0,1,\ldots,  r$ and $n\in\Z_{>0}$. 
By the pigeonhole principle,  there are integers $m,  n\in\Z_{>0}$ with $m<n$ such that 
\[
s_k(\alpha_1^m,  \alpha_2^m,  \ldots,  \alpha_r^m)=s_k(\alpha_1^n,  \alpha_2^n,  \ldots,  \alpha_r^n)\eqqcolon s_k
\]
 for all $k=0,1,\ldots,  r$.
Then both $\{\alpha_j^m\}_{j=1}^r$ and $\{\alpha_j^n\}_{j=1}^r$ are the set of the roots of the polynomial $\sum_{k=0}^r (-1)^ks_kt^{r-k}$,  and hence there exists a permutation $\tau\in\fS_r$ satisfying $\alpha_j^m=\alpha_{\tau(j)}^n$ for any $j=1,2,\ldots,  r$. 
If $d$ denotes the order of $\tau$, then we have 
    \[
    \alpha_j^{m^d} = (\alpha_{\tau(j)}^n)^{m^{d-1}}
    = (\alpha_{\tau^2(j)}^{n^2})^{m^{d-2}} = \cdots 
    = \alpha_{\tau^d(j)}^{n^d} = \alpha_j^{n^d},  
    \]
which implies $\alpha_j$ is a root of unity.

\item[(2)] We may assume that $\alpha\neq\pm2$.
Let $\zeta$ be a root of the polynomial $t^2-\alpha t+1$.
Then $\zeta\in\{(\alpha\pm\sqrt{\alpha^2-4})/2\}$.
It follows from $\alpha^2-4<0$ that $|\zeta|=1$, and the other root of $t^2-\alpha t+1$ is $\zeta^{-1}$.
In particular $\alpha=\zeta+\zeta^{-1}$.
Hence it suffices to show that $\zeta$ is a root of unity.

Note that $\zeta\in\bar{\Z}$ since is $\bar{\Z}$ is integrally closed.
For each $\alpha_j$,  set $\zeta_j=(\alpha_j+\sqrt{\alpha_j^2-4})/2$.
Then all conjugates of $\zeta$ over $\Q$ is $\zeta_1^{\pm},  \zeta_2^{\pm1},  \ldots,  \zeta_r^{\pm1}$.
Since all of them have  absolute value $1$,  it follows from (1) that $\zeta$ is a root of unity.  

\item[(3)] We see from (2) that,  by replacing $\alpha$ with $-\alpha$ if necessary, there exist an embedding $\iota \colon \mathbb{Q}(\alpha) \hookrightarrow \mathbb{R}$ 
and a positive integer $n$ such that $\iota(\alpha) = 2 \cos(\pi/n)$. 
Since $|\iota(\alpha)| \le \sqrt{2}$ by assumption, we deduce that $n \le 4$.
Consequently,  $\alpha \in \{0, \pm 1, \pm \sqrt{2}\}$.
\end{proof}

\begin{corollary}\label{cos1}
For all $\alpha,  \beta\in\Phi(L)$,  we have $\alpha\cdot\beta=2\cos(r\pi)$ with some $r\in\Q$.
\end{corollary}

\begin{proof}
Set $c=\alpha\cdot\beta\in\cO$.
Since $(\alpha\pm\beta)\cdot(\alpha\pm\beta)=4\pm2c$,  for any embedding $\sigma\colon K \hookrightarrow \R$ we have $\sigma(4\pm2c)\geq0$,  \emph{i.e.} $-2\leq\sigma(c)\leq2$.
Hence the assertion follows from \cref{Kronecker} (2).
\end{proof}

We recall another basic fact on cyclotomic units for later use.

\begin{lemma}\label{cyclotomic_unit}
Let $m$ and $k$ be positive integers with $m\geq2$. 
Set
    \[
    R \coloneqq \Z[2\cos(\pi/m)], \quad 
    c_k \coloneqq \frac{\sin(k\pi/m)}{\sin(\pi/m)}. 
    \]
Then $c_k\in R$.
Moreover, if $m$ and $k$ are coprime, then $c_k \in R^\times$. 
\end{lemma}

\begin{proof}
Let $U_n(t)$ be the Chebyshev polynomial of the second kind.
It is a polynomial with integer coefficients of degree $n$, satisfying
    \[
    U_n(\cos \theta) = \frac{\sin((n+1)\theta)}{\sin \theta}.
    \]
Set $q_n(t) \coloneqq U_n(t/2)$.
From the recurrence relation $q_{n+1}(t) = t q_{n}(t)-q_{n-1}(t)$,  we see that $q_n(t)$ is a monic polynomial in $\Z[t]$.
Hence $c_k=q_{k-1}(2\cos(\pi/m))$ is in $R$.

Suppose that $m$ and $k$ are relatively prime.
We then have 
    \[
     c_ k = \frac{\zeta_{2m}^k-\zeta_{2m}^{-k}}{\zeta_{2m}-\zeta_{2m}^{-1}}
     = \zeta_{2m}^{1-k}\frac{1-\zeta_m^k}{1-\zeta_m},  
    \]
which is a unit of $\Z[\zeta_{2m}]$ by \cite[Propoisition 2.8,  Lemma 8.1]{Washington97}.
Hence it follows that $c_k$ is in $R\cap\Z[\zeta_{2m}]^\times=R^\times$.
\end{proof}

We define the Weyl group as usual.
For a non-zero vector $\alpha\in V\setminus\{0\}$,  let $s_\alpha\in\GL(V)$ denote the reflection with respect to the hyperplane $H_\alpha\coloneqq\{v\in V\otimes_K\R \mid \alpha\cdot v=0\}$.
Precisely,  we set 
    \[
    s_\alpha(v) \coloneqq v-2\frac{\alpha\cdot v}{\alpha\cdot \alpha}\alpha,  \qquad v\in V.
    \]

\begin{definition}
The \emph{Weyl group} $W \coloneqq W(L)$ of $L$ is the subgroup of $\GL(V)$ generated by the reflections $\{s_\alpha \mid \alpha\in\Phi(L)\}$.
\end{definition}
Since reflections preserve the symmetric bilinear form on $V$,  the Weyl group $W$ is a subgroup of the orthogonal group $\O(V)$.
In particular $\Phi$ is stable under $W$.

\subsection{Finite Coxeter systems}\label{sec:recollection}

This subsection is devoted to recalling basic facts about finite Coxeter groups. 

\begin{definition}
A group $W$ is called a \emph{Coxeter group} if it has a presentation
    \[
    W = \langle s_1,  s_2,  \ldots,  s_n \mid (s_is_j)^{m_{ij}} = 1 \rangle,  
    \]
 where $m_{ij}\in\Z\sqcup\{\infty\}$ with the property that $m_{jj}=1$ and $m_{ij}=m_{ji}>1$ for $i\neq j$.
 Set $S \coloneqq \{s_1,  s_2,  \ldots,  s_n\}$,  then the pair $(W,  S)$ is called a \emph{Coxeter system}.
 
The \emph{Schl\"afli matrix} of $(W,  S)$ is the symmetric matrix $A(W,S) \coloneqq (a_{ij})$ given by $a_{ij} = -\cos(\pi/m_{ij})$.
\end{definition}

It is known that the finite Coxeter groups are precisely the finite reflection groups.
In this section,  suppose that $(W,  S)$ is a Coxeter system.

\begin{proposition}[\cite{Humphreys12},   Theorem 6.4]
The following are equivalent.
\begin{itemize}\setlength{\itemsep}{5pt}
\item[(i)] $|W|<\infty$.
\item[(ii)] $A(W,  S)$ is positive-definite.
\item[(iii)] $W$ is a finite reflection group.
\end{itemize}
\end{proposition}

The finite Coxeter groups are classified by using the Coxeter-Dynkin diagrams.

\begin{definition}
The \emph{Coxeter-Dynkin diagram} of $(W,  S)$ is a labelled graph associated to $S=\{s_1,  s_2 ,  \ldots,  s_n\}$ as follows.
The vertex set is one-to-one correspondence with $S$.
For distinct $s_i,  s_j\in\Delta$,  the vertices corresponding to $s_i$ and $s_j$ are connected by an edge if and only if $m_{ij}\geq3$,  equivalently $s_i$ and $s_j$ do not commute. 
If $m_{ij}\geq4$,  this edge is labelled with $m_{ij}$.
\end{definition}

Note that two Coxeter systems are isomorphic if and only if they have the same Coxeter-Dynkin diagram. 

A Coxeter group or a Coxeter system is said to be \emph{irreducible} if the corresponding Coxeter-Dynkin diagrams are connected.
Each Coxeter group is uniquely written as a direct product of irreducible Coxeter groups. 
The classification of irreducible finite Coxeter groups is as follows.

\begin{theorem}[\cite{Humphreys12},  Theorem 2.7]\label{Coxeter_classification}
The Coxeter-Dynkin diagrams of irreducible finite Coxeter systems are listed below: 

\noindent\begin{minipage}{0.4\textwidth}
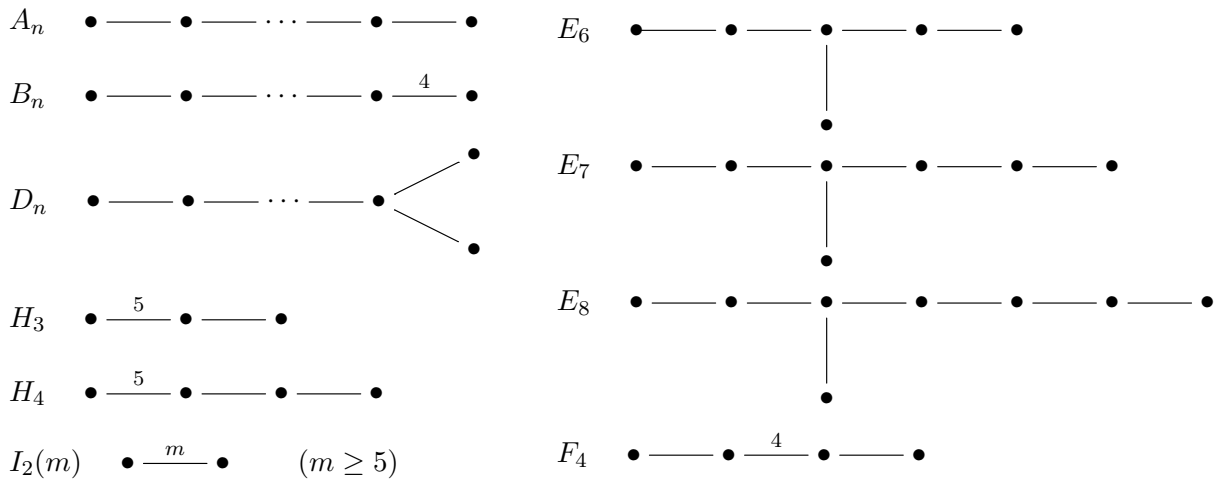
    \begin{align*}
    &A_n \quad \xygraph{
    \bullet  - [r]
    \bullet  - [r] \cdots - [r]
    \bullet  - [r]
    \bullet}   \\[5pt]
    &B_n \quad \xygraph{
    \bullet  - [r]
    \bullet  - [r] \cdots - [r]
    \bullet  -^4 [r]
    \bullet }   \\[5pt]
    &D_n \quad \xygraph{
    \bullet  - [r]
    \bullet  - [r] \cdots - [r]
    \bullet  (
        - []!{+(1,.5)} \bullet ,
        - []!{+(1,-.5)} \bullet )}   \\[5pt]
    &H_3 \quad \xygraph{
    \bullet  -^5 [r]
    \bullet  - [r] 
    \bullet }   \\[5pt]
    &H_4 \quad \xygraph{
    \bullet  -^5 [r]
    \bullet  - [r] 
    \bullet  - [r]
    \bullet }   \\[5pt] 
    &I_2(m) \quad \xygraph{
    \bullet  -^m [r]
    \bullet }  \qquad (m\geq5)
    \end{align*}
\end{minipage}\qquad
\begin{minipage}{0.5\textwidth}
    \begin{align*}
    &E_6 \quad \xygraph{
    \bullet ) - [r]
    \bullet  - [r]
    \bullet  (
        - [d] \bullet ,
        - [r] \bullet 
        - [r] \bullet 
    )}  \\
    &E_7 \quad \xygraph{
    \bullet  - [r]
    \bullet  - [r]
    \bullet  (
        - [d] \bullet ,
        - [r] \bullet 
        - [r] \bullet 
        - [r] \bullet 
    )}  \\
    &E_8 \quad \xygraph{
    \bullet  - [r]
    \bullet  - [r]
    \bullet  (
        - [d] \bullet ,
        - [r] \bullet 
        - [r] \bullet 
        - [r] \bullet 
        - [r] \bullet 
     )}  \\
    &F_4 \quad \xygraph{
    \bullet  - [r]
    \bullet  -^4 [r] 
    \bullet - [r]
    \bullet 
    )} 
    \end{align*}
\end{minipage}\vspace{10pt} 
\\
Here,  $A_n$ \ ($n\geq1$),  $B_n$ \ ($n\geq2$),  and $D_n$ \ ($n\geq4$) have $n$ vertices.
Conversely,  each diagram has the associated finite Coxeter system.
\end{theorem}

\section{Structure of root systems}\label{sec:root system}

Let $K$ be a totally real number field,  $\cO=\cO_K$ its ring of integers,  and $L$ a root lattice over $\cO$.
    \[
    \textbf{Throughout this section,  we assume $|\Phi(L)|<\infty$.}
    \]

\begin{lemma}\label{finite}
Suppose that $[K:\Q]<\infty$.
Then the above condition is satisfied,  \emph{i.e.,} $\Phi(L)$ is a finite set.
\end{lemma}

\begin{proof}
It is well-known that  $\cO$ is a $\Z$-lattice in $K\otimes_\Q\R \simeq \R^{[K : \Q]}$. 
Since $L$ is an $\cO$-lattice in $V$, it follows that $L$ is a discrete subset of   $V \otimes_\Q\R$. 
Consider an $\R$-bilinear form $\langle \cdot,  \cdot\rangle \colon (V \otimes_\Q\R) \times (V \otimes_\Q\R) \to \R$ characterized by 
\[
\langle x , y  \rangle = \mathrm{trace}_{K/\mathbb{Q}}(x \cdot y) 
\]
for any $x,y \in V$. Note that if $x \cdot y \in \Q$, then $\langle x, y \rangle =  [K : \Q](x\cdot y)$. 
Since the bilinear form $(x,y) \mapsto x \cdot y$ on $V$ is totally positive definite, $\langle \cdot,  \cdot\rangle$ is an inner product on $V \otimes_\Q\R$. In particular, the set $\{x \in V \otimes_\Q\R \mid \langle x, x \rangle = 2[K : \Q] \}$ is compact. 
Hence $\Phi(L) \subset L \cap \{x \in V \otimes_\Q\R \mid \langle x, x \rangle = 2[K : \Q] \}$ is finite since it is discrete and compact. 
\end{proof}

\subsection{Fundamental system of roots}\label{sec:fundamental}

We prove that $L$ has a fundamantal system,  following the lines of \cite[\S\,1.4]{Ebeling13}.
By the assumption that $|\Phi(L)| < \infty$, one can take a non-zero vector $t\in V\otimes_K\R=\R^n$ so that $t\cdot\alpha\neq0$ for all $\alpha\in\Phi(L)$.
Here,  we extend the bilinear form on $V$ to $\R^n$ by fixing a basis of $V$.
Set 
    \[
    \Phi^+_t \coloneqq \{\alpha\in\Phi \mid t\cdot\alpha>0\}
    \]
and $\Phi^-_t \coloneqq \{-\alpha \mid \alpha\in\Phi_t^+\}$ so that we have the partition $\Phi=\Phi_t^+\sqcup\Phi_t^-$.

For any $\alpha,  \beta\in\Phi_t^+$ we write $\alpha\preceq\beta$ if there are roots $\beta_1 \coloneqq \beta,  \beta_2,  \ldots,  \beta_r\in\Phi_t^+$ and scalars $c_1,  c_2,  \ldots,  c_r\in \cO_{\geq1}$ such that $\alpha=\sum_{j=1}^rc_j\beta_j$.
This defines a partial order $\preceq$ on $\Phi_t^+$.
Let $\Delta_t$ be the set of maximal elements with respect to the order $\preceq$.

\begin{lemma}\label{fund_negative}
For any distinct $\alpha,  \beta\in\Delta_t$,  we have $\alpha\cdot\beta=-2\cos(\pi/m)$ with some integer $m\geq2$.
In particular $\alpha\cdot\beta\leq0$. 
\end{lemma}

\begin{proof}
Set $\Phi' \coloneqq \Phi\cap (\R\alpha+\R\beta)$. 
We prove that $\Phi'$ consists of vertices of a regular $(2m)$-gon for some $m\geq 2$ in the Euclidean plane $\R\alpha+\R\beta=\R^2$.

Since $|\Phi'| < \infty$ and $\alpha \neq \pm \beta$, we may choose $\alpha_0, \alpha_1 \in \Phi'$ such that the angle $0 < \theta < \pi$  between $\alpha_0$ and $\alpha_1$ is minimal among all angles formed by distinct pairs of roots in $\Phi'$. 
We inductively define $\alpha_{j+2}$ for $j\ge 0$ by 
    \begin{equation}\label{rot}
        \alpha_{j+2} = -s_{\alpha_{j+1}}\alpha_{j}
        = -\alpha_{j}+(\alpha_{j+1}\cdot\alpha_j)\alpha_{j+1}.
    \end{equation}
Since $\Phi'$ is invariant under reflections with respect to its elements, we have $\alpha_j \in \Phi'$ for all $j \ge 0$. 
From the recurrence
\[
\alpha_{j+2} \cdot \alpha_{j+1} 
= - \alpha_j \cdot (s_{\alpha_{j+1}} \alpha_{j+1})
= \alpha_j \cdot \alpha_{j+1},
\]
it follows that $\alpha_{j+1} \cdot \alpha_j = \alpha_1 \cdot \alpha_0 = 2 \cos \theta$ for all  $j \ge 0$.
In other words,  the angle between $\alpha_j$ and $\alpha_{j+1}$ equals $\theta$ for all $j\ge 0$.
Also note that 
    \[
    \alpha_{j+2}\cdot \alpha_j = (\alpha_{j+1}\cdot\alpha_j)^2-\alpha_j\cdot\alpha_j
    = 2(2\cos^2\theta-1) = 2\cos2\theta,
    \]
which in particular implies that $\alpha_{j+2}$ is different from $\alpha_j$. 
Since $|\Phi'|<\infty$,  there exists a positive integer $k$ such that $\alpha_k=\alpha_0$.
Let $N$ be the smallest such positive  integer. 
From the above argument,  it follows that roots $\alpha_0,\alpha_1,\ldots,  \alpha_{N-1}$ are adjacent vertices of a regular $N$-gon. 

Suppose that $\gamma \in \Phi'$ is different from all of these vertices. 
Then there is a unique index $j_0$ such that both the angle between $\gamma$ and $\alpha_{j_0}$ and the angle between $\gamma$ and $\alpha_{j_0+1}$ are smaller than $\theta$.
This contradicts the minimality of $\theta$. 
Hence we obtain
\[
\Phi' = \{\alpha_0, \alpha_1, \ldots, \alpha_{N-1}\}.
\]
Since $\Phi'$ is closed under the multiplication by $(-1)$,  it follows that $N$ is even. 
Therefore, $\Phi'$ consists of vertices of a regular $(2m)$-gon with $m\coloneqq\pi/\theta\ge 2$ and $N=2m$. 

From the definition of $\Phi_t^+$, by replacing the adjacent roots $\alpha_0, \alpha_1$ if necessary, we may assume that $\Phi'\cap \Phi_t^+=\{\alpha_0,\ldots,\alpha_{m-1}\}$. 
In particular, $\alpha=\alpha_k$ for some $0\le k< m$. Now $\alpha_k$ is written as
    \[
    \alpha_{k} = \frac{\sin((1+k)\pi/m)}{\sin(\pi/m)}\alpha_0
    +\frac{\sin(k\pi/m)}{\sin(\pi/m)}\alpha_{m-1}.
    \]
Remark that the coefficients of $\alpha_0,\alpha_{m-1}$ belong to $\cO$ by \cref{cyclotomic_unit}, and to $\R_{\ge 1}$ if $0<k<m-1$. 
Hence the maximality of $\alpha\in\Phi_t^+$ implies $k=0,m-1$, and $\alpha=\alpha_0,\alpha_{m-1}$. 
Similarly, we obtain $\beta=\alpha_0,\alpha_{m-1}$. 
From $\alpha\neq \beta$, we see $\alpha\cdot\beta=\alpha_0\cdot\alpha_{m-1}=-2\cos(\pi/m)$. 
\end{proof}

\begin{lemma}\label{fund_lin_indep}
The set $\Delta_t$ is linearly independent over $K$.
\end{lemma}

\begin{proof}
Suppose that  $\sum_{\alpha\in\Delta_t}c_\alpha\alpha=0$ for some $c_\alpha\in K$.
Set $S_\pm \coloneqq \{\alpha\in\Delta_t \mid \pm c_\alpha>0\}$ to obtain $\sum_{\beta\in S_+}c_\beta\beta =\sum_{\gamma\in S_-}(-c_\gamma)\gamma$.
Write the left-hand side of this equality as $\lambda$.
Then it follows from \cref{fund_negative} that 
    \[
   \lambda \cdot \lambda 
    = \sum_{\beta\in S_+}\sum_{\gamma\in S_-}c_\beta(-c_\gamma)(\beta\cdot\gamma)
    \leq 0. 
    \]
Hence we get $\lambda=0$ and
    \[
    0 = t\cdot\lambda = \sum_{\beta\in S_+}c_\beta(t\cdot\beta).
    \]
Since $c_\beta > 0$ and $t \cdot \beta > 0$ for each $\beta \in S_+$, we deduce that $S_+ = \emptyset$. Similarly, $S_- = \emptyset$, which implies that $c_\alpha = 0$ for all $\alpha \in \Delta_t$. 
\end{proof}

\begin{corollary}\label{base1}
\begin{itemize}\setlength{\itemsep}{5pt}
\item[(1)] The set $\Delta_t$ is a fundamental system of roots for $L$.
In particular,  $L$ has a fundamental system of roots.
\item[(2)] If $\Delta$ is a fundamental system of roots for $L$,  then there exists $t\in\R^n$ such that $\Delta=\Delta_t$.
\end{itemize}
\end{corollary}

\begin{proof}
\item[(1)] From the definition of $\Delta_t$,  each root in $\Phi_t^+$ is written in the form $\sum_{\alpha\in\Delta_t}c_\alpha\alpha$ with $c_\alpha\in \cO_{\geq1}$.
In particular, $\Delta_t$ spans $V$. 
Combined with \cref{fund_lin_indep}, this shows that $\Delta_t$ is a fundamental system of roots. 

\item[(2)] Let $\Phi^+$ be  the set of roots which are non-negative coefficient linear combination of the roots in $\Delta$.
Set $\Phi^- \coloneqq \{-\beta \mid \beta\in\Phi^+\}$.
This gives the partition $\Phi = \Phi^+ \sqcup \Phi^-$.  
Take a non-zero vector $t\in\R^n$ so that $t\cdot\alpha>0$ for all $\alpha\in\Delta$.
It follows that $\Phi^{\pm} \subset \Phi_t^{\pm}$, and hence $\Phi^{\pm} = \Phi_t^{\pm}$. 
Since maximal elements in $\Phi_t^+$ belong to $\Delta$, we obtain $\Delta_t\subset\Delta$. 
Then the equality of cardinalities $|\Delta| = n = |\Delta_t|$ implies $\Delta = \Delta_t$.
\end{proof}

\subsection{Irreducibility of root lattices}\label{sec:Irreducibility}

We say that $\Phi$ is \emph{reducible} if there exists a disjoint decomposition $\Phi=\Phi_1\sqcup\Phi_2$ with $\Phi_1\neq\emptyset$ and $\Phi_2\neq\emptyset$ such that  $\alpha_1\cdot\alpha_2=0$ for any $\alpha_1\in\Phi_1$ and $\alpha_2 \in \Phi_2$. 
Otherwise $\Phi$ is said to be \emph{irreducible}.

\begin{lemma}\label{lem:reducibility}
Suppose that $L$ is reducible as an $\cO$-lattice, i.e., there exists an orthogonal decomposition $L = L_1 \oplus L_2$.
Set $\Phi_i \coloneqq \Phi \cap L_i$ for each $i = 1,2$.
\begin{itemize}
\item[(1)] We have a disjoint decomposition $\Phi = \Phi_1 \sqcup \Phi_2$.
\item[(2)] For each $i = 1,2$, we have $\Phi(L_i) = \Phi_i$; in particular, $L_i$ is a root lattice over $\cO$.
\end{itemize}
Consequently, every root lattice over $\cO$ admits a unique decomposition as an orthogonal direct sum of irreducible root lattices.
\end{lemma}

\begin{proof}
\item[(1)] Let $\alpha \in \Phi$, and write $\alpha = \alpha_1 + \alpha_2$ with $\alpha_1 \in L_1$ and $\alpha_2 \in L_2$.
Assume that $\alpha_2 \neq 0$. Let us show that $\alpha_1 = 0$. 
For any embedding $\sigma  \colon K \hookrightarrow \R$,  we have 
    \[
    2 = \sigma(\alpha\cdot\alpha)
    = \sigma(\alpha_1\cdot\alpha_1)+\sigma(\alpha_2\cdot\alpha_2),  
    \]
which implies that $0\leq \sigma(\alpha_1\cdot\alpha_1) <  2$ since $\alpha_2 \neq 0$. 
It follows from \cref{Kronecker}(2) that we can find relatively prime integers $s$ and $t > 1$ with $\alpha_1 \cdot \alpha_1 = 2\cos(s\pi/t)$.  
If $t$ and $s$ are odd, then $2t$ is coprime to $s$ and $t-2$. 
If $t$ is odd and $s$ is even, then $t$ is coprime to $s/2$ and $(t-1)/2$. 
If $t$ is even, then $s$ is odd and $2t$ is coprime to $s$ and $t-1$. 
Hence, there is an embedding $\sigma \colon K \hookrightarrow \R$ such that 
\[
\sigma(\cos(s\pi/t)) = 
\begin{cases}
\cos((t-2)\pi/t) = -\cos(2\pi/t) & \textrm{if $t$ and $s$ are odd},  
\\
\cos((t-1)\pi/t) = -\cos(\pi/t) & \textrm{otherwise}. 
\end{cases}
\]  
Therefore, the inequality $0\le \sigma(\cos(s\pi/t)) < 1$ 
implies that $t = 2$ or $t = 3$. 
If $t = 3$, we have $\alpha_1 \cdot \alpha_1 = 1$.  
Since $L$ is generated by $\Phi$, it follows that $\alpha_1 \cdot \alpha_1 \in 2\mathcal{O}$, which is a contradiction. 
Hence, we must have $t = 2$, and consequently $\alpha_1 = 0$.  
\item[(2)] Obviously, $\Phi(L_i) \subset \Phi \cap L_i = \Phi_i$, and hence $\Phi_i = \Phi(L_i)$ by (1).  
As $L$ is generated by $\Phi$, each $L_i$ is generated by $\Phi_i$, and in particular, $L_i$ is a root lattice.
\end{proof}

\begin{corollary}\label{cor:reducibility}
A root lattice $L$ is irreducible if and only if $\Phi$ is irreducible.
\end{corollary}

\begin{proof}
Suppose that $\Phi$ is reducible, that is, there exists a disjoint decomposition $\Phi=\Phi_1\sqcup\Phi_2$ with $\Phi_1\neq\emptyset$ and $\Phi_2\neq\emptyset$ such that $\alpha_1\cdot\alpha_2=0$ for any $\alpha_1 \in\Phi_1$ and   $\alpha_2 \in\Phi_2$. 
For each $i=1,  2$, let $L_i$ denote the sublattice of $L$ generated by $\Phi_i$. 
Since $L$ is generated by $\Phi$,  we have the orthogonal decomposition $L=L_1\oplus L_2$. Hence $L$ is reducible.
This proves the `only if' part. 
The other direction immediately follows from \cref{lem:reducibility}.
\end{proof}

Let $\widetilde{K}$ be a totally real field which is a finite extension of $K$.
The ring of integers of $\widetilde{K}$ is denoted by $\widetilde{\cO}$ and set $\widetilde{L} \coloneqq L\otimes_\cO\widetilde{\cO}$.

\begin{proposition}\label{irred_extension}
The root lattice $\widetilde{L}$ over $\widetilde{\cO}$ is irreducible if and only if $L$ is irreducible.
\end{proposition}

\begin{proof}
Since $\widetilde{\cO}$ is a flat $\cO$-module,  when $L$ is reducible, $\tilde{L}$ is also reducible. 
Next, suppose that $L$ is irreducible.  
To show  that $\widetilde{L}$ is irreducible, assume that there exists a disjoint decomposition $\widetilde{\Phi} = \widetilde{\Phi}_1 \sqcup \widetilde{\Phi}_2$ such that $\alpha_1 \cdot \alpha_2 = 0$ for any $\alpha_1 \in \widetilde{\Phi}_1$ and $\alpha_2 \in \widetilde{\Phi}_2$. 
Since $\Phi$ is irreducible by \cref{cor:reducibility} and $\Phi \subset \widetilde{\Phi}$, we have $\Phi \subset \widetilde{\Phi}_i$ for some $i$.  
As $\widetilde{L}$ is generated by $\Phi$ (as an $\widetilde{\cO}$-module), it is generated by $\widetilde{\Phi}_i$. 
Thus $\widetilde{\Phi} = \widetilde{\Phi}_i$,  and it follows from \cref{cor:reducibility} that $\widetilde{L}$ is irreducible.
\end{proof}

\subsection{Comparison with Coxeter systems}\label{sec:Comparison}

Now we fix a fundamental system of roots $\Delta$.
For each $\alpha, \beta \in \Delta$, denote by $m(\alpha, \beta)$ the positive integer $m$ satisfying $\alpha \cdot \beta = -2\cos(\pi/m)$ (see \cref{fund_negative}). 
Note that $m(\alpha,  \beta)=1$ if and only if $\alpha=\beta$.

\begin{definition}\label{Coxeter-Dynkin}
The \emph{Coxeter-Dynkin diagram} of $L$ is a labelled graph associated with $\Delta$ as follows.
The vertex set is one-to-one correspondence with $\Delta$.
For distinct $\alpha,  \beta\in\Delta$,  the vertices corresponding to $\alpha$ and $\beta$ are connected by an edge if and only if $m(\alpha,  \beta)\geq3$,  which is equivalent to that $\alpha\cdot\beta\neq0$. 
If $m(\alpha,  \beta)\geq4$,  this edge is labelled with $m(\alpha,  \beta)$.
\end{definition}

Note that the Coxeter-Dynkin diagram is connected if and only if $\Phi$ is irreducible.
From \cref{cor:reducibility},  we see that $L$ is irreducible if and only if the Coxeter-Dynkin diagram is connected.  
Write $\Delta=\{\alpha_1,  \alpha_2,  \ldots,  \alpha_n\}$ and set $m_{ij} \coloneqq m(\alpha_i,  \alpha_j)$ for any integers $i,  j \in \{1,2,\ldots, n\}$.

\begin{definition}
The \emph{Schl\"afli matrix} of $L$ is a symmetric $n\times n$ matrix $A(L) \coloneqq (a_{ij})$ given by
    \[
    a_{ij} \coloneqq  -\cos(\pi/m_{ij}) = \frac12\alpha_i\cdot\alpha_j.
    \]
 Note that $2A(L)$ is the Gram matrix of $\Delta$. In particular, $A(L)$ is (totally) positive definite.
\end{definition}

\begin{proposition}\label{injection}
Suppose that $L$ is irreducible.
The Coxeter-Dynkin diagram of $L$ is one of the graphs listed in \cref{Coxeter_classification}.
\end{proposition}

\begin{proof}
Since $A(L)$ is positive definite,  the Coxeter-Dynkin diagram of $L$ is positive definite in the sense of \cite[\S\,2.3]{Humphreys12}.
Such diagrams are classified by \cite[Theorem 2.7]{Humphreys12}.
\end{proof}

\subsubsection{}\label{section_obs}
We summarize some immediate observations.
\begin{itemize}
\item The root system of a finite reflection group realized in the Euclidean space $ V\otimes_K\R$ is a finite set of non-zero vectors $\Psi\subset V\otimes_K\R$ satisfying the following conditions (\cite[\S\,1.2]{Humphreys12}):
\begin{itemize}
\item[(R1)] $\Psi\cap\R\alpha=\{\alpha,  -\alpha\}$ for all $\alpha\in\Psi$;
\item[(R2)] $s_\alpha\Psi=\Psi$ for all $\alpha\in\Psi$.
\end{itemize}
Obviously the set $\Phi(L)$ satisfies these conditions. 
Note that the root system of a finite reflection group is unique up to the scaling of each root. 
\item The Weyl group $W(L)$ of the root lattice $L$ coincides with the finite Coxeter group corresponding to the root system $\Phi(L)$, that is, the group generated by reflections with respect to elements in $\Phi(L)$.
\item Given a non-zero vector $t\in V\otimes_K\R$ so that $t\cdot\alpha\neq0$ for all $\alpha\in\Phi(L)$,  then there exists a unique simple system $\Delta'_t\subset\Phi(L)$ so that $\Phi^+_t$ is the positive system in the context of finite reflection groups \cite[\S\,1.3]{Humphreys12}.
Comparing the definition in \cite[\S\,1.3]{Humphreys12} with \cref{def_fund},  it follows that $\Delta_t\subset\Delta'_t$. 
However,  since both $\Delta'_t$ and $\Delta_t$ are basis of $V\otimes_K\R$,  we obtain $\Delta_t=\Delta'_t$.
This proves that our notion of fundamental system of roots coincides with simple system of finite reflection groups.
\end{itemize}
Now we can appeal to the known results on finite reflection groups to conclude the following propositions.

\begin{proposition}\label{indep}
The Weyl group $W(L)$ acts simply transitively on the set of fundamental systems of roots of $L$.
In particular,  the Coxeter-Dynkin diagram of $L$ is independent of the choice of a fundamental system of roots.
\end{proposition}
\begin{proof}
See \cite[Theorems 1.4 and 1.8]{Humphreys12} for example.
\end{proof}

\cref{injection} and \cref{indep} show that there exists a natural injection from the set of isomorphism classes of irreducible root lattices over $\cO$ to the set of irreducible finite Coxeter systems.
A root lattice is said to be \emph{of type $X_n$} if the corresponding Coxeter-Dynkin diagram is of type $X_n$.

\section{
The case of rank greater than $2$
}\label{sec:scalar}

Suppose that $\widetilde{K}/K$ is an extension of totally real fields. Remark that we do not assume $[\widetilde{K}:K]<\infty$. 
The ring of integers of $K$ and $\widetilde{K}$ are denoted by $\cO$ and $\widetilde{\cO}$,  respectively.
Let $L$ be an irreducible root lattice over $\cO$ and set $\widetilde{L} \coloneqq L\otimes_\cO\widetilde{\cO}$.
Note that $\widetilde{L}$ is an irreducible root lattice over $\widetilde{\cO}$ due to \cref{irred_extension}.
In this section,  we prove the next proposition.

\begin{proposition}\label{lem:extension1}
Suppose that $\rank(L)\geq3$ and $|\Phi(L)|<\infty$.
Then we have $\Phi(L)=\Phi(\widetilde{L})$.
In particular,  we have $|\Phi(\widetilde{L})|<\infty$ and the type of $\widetilde{L}$ is the same as that of $L$.
\end{proposition}

Since $\Phi(L) \subset \Phi(\widetilde{L})$ is obvious, we prove the converse inclusion by a case-by-case analysis.

\subsection{The case of type $A_n$}\label{sec:typeA}

Suppose that $L$ is of type $A_n$ with $n\geq3$.
Let $\Delta \coloneqq \Delta(L)$ be a fundamental system of roots and write $\Delta=\{\alpha_1, \ldots, \alpha_n\}$ so that we have 
    \[
    \alpha_i\cdot\alpha_j = 
        \begin{cases}
         2 & \text{if $i=j$},  \\ 
        -1 & \text{if $|i-j|=1$},  \\ 
         0 & \text{if $|i-j| \geq 2$}. 
        \end{cases}
    \]
Take $\beta \in \Phi(\widetilde{L})$ and write $\beta = r_1\alpha_1 + \cdots + r_n \alpha_n$ with $r_1, \ldots, r_n \in \widetilde{\cO}$. 
The condition $\beta\cdot\beta = 2$ is equivalent to 
    \begin{align}\label{eq:A_n}
    r_1^2 + (r_1 - r_2)^2 + \cdots + (r_{n-1} - r_n)^2 + r_n^2 = 2.     
    \end{align}
For any embedding $\iota \colon \widetilde{K} \hookrightarrow \R$, we obtain
    \[
    \iota(r_1)^2 + \iota(r_1 - r_2)^2 + \cdots + \iota(r_{n-1} - r_n)^2 + \iota(r_n)^2 = 2,  
    \]
which implies
    \begin{gather*}
    |\iota(r_1)| \leq \sqrt{2}, \quad |\iota(r_n)| \leq \sqrt{2}, \quad 
    |\iota(r_i - r_{i+1})| \leq \sqrt{2}, \quad 
    \text{for $i=1, 2,  \ldots,  n-1$.} 
    \end{gather*}
Hence it follows from \cref{Kronecker} (3) that 
    \[
    r_1, r_1-r_2, \ldots, r_{n-1} - r_n, r_n \in \{0, \pm 1, \pm \sqrt{2}\}. 
    \]
If one of $r_1, r_1 - r_2, \ldots, r_{n-1} - r_n, r_n$ is equal to $\pm \sqrt{2}$,  then the equation \eqref{eq:A_n} forces all the remaining ones to be $0$. 
However,  if all but one of them are $0$,  then the remaining one must also be $0$. 
This shows that none of them can be equal to $\pm \sqrt{2}$. 
Consequently,  we have $r_1, \ldots, r_n \in \Z$, and thus $\beta \in \Phi(L)$.
This completes the proof of \cref{lem:extension1} for the case of type $A_n$, and we conclude that 
the set of roots coincides with the standard root system of type $A_n$: 
\begin{align}\label{A}
	\Phi(L) = \{ \pm(\alpha_i+\alpha_{i+1}+\cdots+\alpha_j)
    \mid 1\leq i\leq j\leq n\}.
\end{align}

\begin{corollary}\label{existence_An}
For any  totally real field $K$,  there is an irreducible root lattice over $\cO$ of type $A_n$ with $n\geq3$. 
\end{corollary}

\begin{proof}
When $K = \mathbb{Q}$, 
the above discussion implies that the set of roots $\Phi(L)$ of the root lattice $L=\sum_{i=1}^n\Z\alpha_i$ equals \eqref{A}. 
Therefore $L$ is an irreducible root lattice over $\Z$ of type $A_n$.
For a general totally real field $K$, we then obtain a root lattice of type $A_n$ by extending scalars from $\mathbb{Z}$ to $\mathcal{O}_K$, that is, $L \otimes_{\mathbb{Z}} \mathcal{O}_K$. 
\end{proof}

\subsection{The case of type $B_n$}\label{sec:typeB}

Suppose that $\sqrt{2}\in K$ and $L$ is of type $B_n$ with $n\geq3$. 
Let $\Delta \coloneqq \Delta(L)$ be a fundamental system of roots and write $\Delta=\{\alpha_1, \ldots, \alpha_n\}$ so that we have 
    \[ 
        \begin{cases}
         \alpha_i\cdot\alpha_i = 2,   \\ 
        \alpha_i\cdot\alpha_{i+1} =-1 \ \text{for $i=1,  2, \ldots,  n-2$},  \\ 
        \alpha_{n-1}\cdot\alpha_n =-\sqrt{2},    \\
         \alpha_i\cdot\alpha_j =0  \ \text{if $|i-j| \geq 2$}. 
        \end{cases}
    \]
Take $\beta \in \Phi(\widetilde{L})$ and write $\beta = r_1\alpha_1 + \cdots + r_n \alpha_n$ with $r_1, \ldots, r_n \in \widetilde{\cO}$. 
The condition $\beta\cdot\beta=2$ is equivalent to
    \begin{align}\label{eq:B_n}
    r_1^2 + (r_1 - r_2)^2 + \cdots + (r_{n-2} - r_{n-1})^2 + (r_{n-1} - \sqrt{2} r_n)^2 = 2.     
    \end{align}
Arguing as in the case of type $A_n$,   we obtain from \cref{Kronecker} (3) that 
    \[
    r_1, r_1 - r_2, \ldots, r_{n-2} - r_{n-1},  r_{n-1} - \sqrt{2}r_n \in \{0, \pm 1, \pm \sqrt{2}\}.
    \]
Hence we have $r_1, \ldots, r_n \in \Q[\sqrt{2}] \cap \widetilde{\cO} = \Z[\sqrt{2}] \subset \cO$,  which implies $\beta \in \Phi(L)$.
This completes the proof of \cref{lem:extension1} for the case of type $B_n$. 
An easy calculation shows that the set of roots is as follows:
    \[
    \Phi(L)=\left\{ 
        \begin{array}{c}
        \pm(\alpha_i+\alpha_{i+1}+\cdots+\alpha_{j-1}),  \\
        \pm(\sqrt{2}\alpha_k+\sqrt{2}\alpha_{k+1}+\cdots+\sqrt{2}\alpha_{n-1}+\alpha_n),  \\
        \pm(\alpha_i+\cdots+\alpha_{j-1}+2\alpha_j+\cdots+2\alpha_{n-1}+\sqrt{2}\alpha_n)
        \end{array}
    \, \middle|\, 
    \begin{array}{c}
    1\leq i<j\leq n,\\ 
    1\leq k\leq n
    \end{array}
    \right\}.
    \]

The next corollary follows from the same argument as \cref{existence_An}.
\begin{corollary}\label{existence_Bn}
For any  totally real field $K$ containing $\sqrt{2}$,  there exists an irreducible root lattice over $\cO$ of type $B_n$ with $n\geq3$. 
\end{corollary}

\subsection{The case of type $D_n$}\label{sec:typeD}

Suppose that $L$ is of type $D_n$ with $n\geq4$.
Note that the root lattice of type $D_3$ is isomorphic to that of type $A_3$.

Let $\Delta=\Delta(L)$ be a fundamental system of roots and write $\Delta=\{\alpha_1, \ldots, \alpha_n\}$ so that we have for any $1 \leq i \leq j \leq n$, we have 
    \begin{align*}
    \alpha_i\cdot\alpha_j = 
        \begin{cases}
        2 & \textrm{if } i=j,   \\
        -1 & \textrm{if } j = i+1 \leq n -1,  \\
        -1 & \textrm{if } i = n-2 \text{ and } j = n,  \\
        0 & \textrm{otherwise}. 
        \end{cases}
    \end{align*}

Take $\beta \in \Phi(\widetilde{L})$ and write $\beta = r_1\alpha_1 + \cdots + r_n \alpha_n$ with $r_1, \ldots, r_n \in \widetilde{\cO}$. 
The condition $\beta\cdot\beta=2$ is equivalent to
    \begin{align}\label{eq:D_n}
    r_1^2 + (r_1 - r_2)^2 + \cdots + (r_{n-3} - r_{n-2})^2 
    + \frac{1}{2}(r_{n-2} - 2r_{n-1})^2 + \frac{1}{2}(r_{n-2} - 2r_n)^2 = 2.     
    \end{align}
By the same argument as in the case of type $A_n$,  it follows from  \cref{Kronecker} (3) that
    \[
    r_1, r_1 - r_2, \ldots, r_{n-3} - r_{n-2} \in \{0, \pm 1, \pm \sqrt{2}\}.
    \]
If one of $r_1, r_1 - r_2, \ldots, r_{n-3} - r_{n-2}$ is equal to $\pm \sqrt{2}$,  then the equation \eqref{eq:D_n} forces all the remaining ones to be $0$. 
Such a situation occurs only when 
$r_{n-2}=\pm\sqrt{2}$. 
Moreover, since we must have $r_{n-2} - 2 r_{n-1} = 0$,  it follows that $r_{n-1} = \pm \sqrt{2}/2$.
This contradicts the assumption that $r_{n-1}$ is an algebraic integer. 
Hence,  we conclude that $r_1, r_1 - r_2, \ldots, r_{n-3} - r_{n-2} \in \{0, \pm 1\}$. 
Moreover, at most two of $r_1,\, r_1 - r_2,\, \ldots,\, r_{n-3} - r_{n-2}$ are nonzero.

\underline{Case 1.} If $r_1 = r_1 - r_2 = \cdots = r_{n-3} - r_{n-2} = 0$, then we have $r_1 = \cdots = r_{n-2} = 0$, and the equation \eqref{eq:D_n} reduces to
    \[
    r_{n-1}^2 + r_n^2 = 1.
    \]
From \cref{Kronecker} we see that $r_{n-1}, r_{n} \in \{0,\pm1\}$.
Hence $r_1,  r_2,  \ldots,  r_n\in\Z$ and in particular $\beta \in \Phi(L)$. 

\underline{Case 2.} If exactly one of $r_1,\, r_1 - r_2,\, \ldots,\, r_{n-3} - r_{n-2}$ is nonzero,  then we have $r_{n-2} = \pm 1$ in any case.
The equation~\eqref{eq:D_n} becomes
    \[
    (r_{n-2} - 2 r_{n-1})^2 + (r_{n-2} - 2 r_n)^2 = 2.
    \]
By the same argument as in the case of type $A_n$,  from  \cref{Kronecker} (3) we obtain
    \[
    r_{n-2} - 2 r_{n-1},\; r_{n-2} - 2 r_n \in \{0, \pm 1, \pm \sqrt{2}\}.
    \]
Since $r_{n-2}=\pm1$ and $r_{n-1},  r_n$ are algebraic integers,  both $r_{n-2} - 2 r_{n-1}$ and $r_{n-2} - 2 r_n$ are $\pm1$.
This proves that $r_1,  r_2,  \ldots,  r_n\in\Z$ and in particular $\beta \in \Phi(L)$.

\underline{Case 3.} If exactly two of $r_1, r_1 - r_2, \ldots, r_{n-3} - r_{n-2}$ is nonzero, then  it follows from the equation \eqref{eq:D_n} that
    \[
    (r_{n-2} - 2 r_{n-1})^2 + (r_{n-2} - 2 r_n)^2  = 0.
    \]
Hence,  $r_{n-2} - 2 r_{n-1} = r_{n-2} - 2 r_{n} = 0$.
Again we obtain $r_1,  r_2,  \ldots,  r_n\in\Z$ and in particular $\beta \in \Phi(L)$. 
This completes the proof of \cref{lem:extension1} for the case of type $D_n$.

The set of roots coincides with the ordinary root system of type $D_n$:
    \[
    \Phi(L)=\left\{ 
        \begin{array}{c}
        \pm(\alpha_i+\alpha_{i+1}+\cdots+\alpha_{j-1}),  \\ 
        \pm(\alpha_k+\cdots+\alpha_{l-1}+2\alpha_l+\cdots+2\alpha_{n-2}+\alpha_{n-1}+\alpha_n) \\
        \pm(\alpha_m+\cdots+\alpha_{n-2}+\alpha_n)
        \end{array}
    \,\middle|\, 
        \begin{array}{c}
        1\leq i<j\leq n+1, \\
        1\leq k<l\leq n-2, \\
        1\leq m\leq n-2
        \end{array} 
    \right\}.
    \]

Similarly as \cref{existence_An} we obtain the next corollary.
\begin{corollary}\label{existence_Dn}
For any  totally real field $K$,  there is an irreducible root lattice over $\cO$ of type $D_n$ with $n\geq4$. 
\end{corollary}

\subsection{The case of type $E_n$,  $F_4$ and $H_n$}\label{sec:exceptional}

Suppose that $L$ is of type  $X_n \in \{E_6, E_7, E_8, F_4, H_3, H_4\}$ and that $\sqrt{2}\in K$ when $X_n=F_4$ and $\sqrt{5}\in K$ when $X_n=H_n$.
As observed in  \cref{section_obs}, \( \Phi(L) \) coincides with the root system of the finite reflection group of type \(X_n\). In particular, we know its cardinality as given in the following table: 
\begin{table}[htb]
\renewcommand{\arraystretch}{1.3}
\begin{tabular}{|c||c|c|c|c|c|c|c|c|c|} \hline  \rule{0mm}{5mm}
$X_n$ & \ $A_n$ & \ $B_n$ & \ $D_n$ & \ $E_6$ \ & \ $E_7$  \ & \ $E_8$ \ & \ $F_4$ \ & \ $H_3$\ & \ $H_4$ \ \\ \hline\hline
$|\Phi(L)|$ & \ $n(n+1)$ & \ $2n^2$ & \ $2n(n-1)$ & $72$  & $126$ & $240$ & $48$ & $30$ & $120$  \\  \hline
\end{tabular} 
\end{table}

Note that $\widetilde{L}$ is a root lattice over $\widetilde{K}$ of rank $n$ with $\Phi(L)\subset\Phi(\widetilde{L})$.
Taking the cardinality of root systems into consideration (see the table above),  the possible types of $\widetilde{L}$ are as follows:
    \[
        \begin{cases}
        B_6 \text{ or }  E_6 & \text{ if $X_n=E_6$},\\
        E_7 & \text{ if $X_n=E_7$},\\
        E_8 & \text{ if $X_n=E_8$},\\
        F_4 \text{ or } H_4 & \text{ if $X_n=F_4$},\\
        H_3 & \text{ if $X_n=H_3$},  \\
        H_4 & \text{ if $X_n=H_4$}.
        \end{cases}
    \]

If $L$ is of type $E_6$, then  we obtain   from the above argument that $|\Phi(L)|=72=|\Phi(\widetilde{L})|$. 
Since $\Phi(L)\subset\Phi(\widetilde{L})$,  it follows that $\Phi(L)=\Phi(\widetilde{L})$. Hence $\widetilde{L}$ is of type $E_6$. 

Suppose that $L$ is of type $F_4$.
Then there exists a pair of roots $\alpha,  \beta\in\Phi(L)\subset\Phi(\widetilde{L})$ such that $\alpha\cdot\beta=-\sqrt{2}$.
This is impossible if $\widetilde{L}$ is of type $H_4$.
Hence $\widetilde{L}$ is of type $F_4$ and $\Phi(L)=\Phi(\widetilde{L})$.

Therefore, in each case we have $\Phi(L)=\Phi(\widetilde{L})$. 
This concludes the proof of \cref{lem:extension1}. 

As for the existence of root lattices of type $E_n$ $(n=6,7,8)$, $F_4$, $H_n$ $(n=3,4)$, we have the following

\begin{corollary}\label{existence_exceptional}
Let the notation be as above.
\begin{itemize}
\item[(1)] For any  totally real field $K$,  there exists an irreducible root lattice over $\cO$ of type $E_n$ with $n=6,  7,  8$. 
\item[(2)] If $c(F_4)=\sqrt{2}\in K$,  there exists an irreducible root lattice over $\cO$ of type $F_4$. 
\item[(3)] If $c(H_3)=c(H_4)=(1+\sqrt{5})/2\in K$,  there exists an irreducible root lattice over $\cO$ of type $H_n$ with $n=3,  4$. 
\end{itemize}
\end{corollary}
\begin{proof}
Let $X_n\in \{E_6,E_7,E_8,F_4,H_3,H_4\}$. 
From the assumption $c(X_n)\in K$, we can 
define a totally positive definite symmetric bilinear form on a free $\cO$-module $L=\bigoplus_{i=1}^n\cO\alpha_i$ so that 
$\alpha_i\cdot\alpha_j=-2\cos(\pi/m(\alpha_i,\alpha_j))$ $(1\leq i,j \leq n, m(\alpha_i,\alpha_j)\in\Z_{> 0})$, $m(\alpha_i,\alpha_i)=1$ $(1\leq i\leq n)$ and the diagram obtained from $\{\alpha_i\}_{i=1}^n$ by the method in \cref{Coxeter-Dynkin} is the Coxeter-Dynkin diagram of type $X_n$. 
Then $L$ becomes an irreducible root lattice over $\cO$ with $\{\alpha_i\}_{i=1}^n\subset \Phi(L)$. 
Now the above discussion in this subsection implies $\{\alpha_i\}_{i=1}^n= \Phi(L)$. Therefore the root lattice $L$ is of type $X_n$. 
\end{proof}

\subsection{Summary for the case of rank greater than $2$}\label{sec:summary_rank>2}

\begin{theorem}\label{summary>2}
Let $K$ be a totally real field. 
Then for any irreducible root lattice $L$ over $\cO$ with $\rank(L)\geq3$, we have $|\Phi(L)|<\infty$,  and the isomorphism class of $L$ is determined by its type $X_n$. 
Moreover, if $n\geq 3$, then  for each type $X_n$, there exists an irreducible root lattice over  $\cO$ of type $X_n$ if and only if $c(X_n)\in K$,  where
     \[
    c(X_n) \coloneqq
        \begin{cases}
        2 & \text{if $X_n=A_n$ ($n\geq3$),  $D_n$ ($n\geq 4$),  or $E_n$ ($n=6,  7,  8$)},  \\
        \sqrt{2} & \text{if $X_n=B_n$ ($n\geq3$) or $F_4$},  \\
        (1+\sqrt{5})/2 & \text{if $X_n=H_n$ ($n=3,  4$).}
        \end{cases}
    \]
\end{theorem}
\begin{proof}
If we prove $|\Phi(L)|<\infty$, then the first assertion follows from \cref{lem:extension1} and the discussion in \cref{sec:Comparison} (remark that $|\Phi(L)|<\infty$ is assumed in \cref{sec:root system}), and the second assertion follows from \cref{existence_An}, \cref{existence_Bn},  \cref{existence_Dn}, and \cref{existence_exceptional}. 

Let us prove $|\Phi(L)|<\infty$. 
Since $L$ is finitely generated as an $\cO$-module and is generated by $\Phi(L)$, it is generated by finitely many roots $\alpha_1,\ldots,\alpha_d\in\Phi(L)$. 
Define $K_0=\Q(\alpha_i\cdot\alpha_j\mid 1\le i,j\le d)$. Then $K_0$ is a totally real number field satisfying $[K_0:\Q]<\infty$. Write $\cO_0$ for the ring of integers of $K_0$.  
Then $L_0:=\sum_{i=1}\cO_0\alpha_i$ is a root lattice over $\cO_0$. 
Moreover, $L_0$ has finitely many roots by \cref{finite} and is irreducible by \cref{irred_extension}. 
Since \cref{lem:extension1} implies $\Phi(L_0)=\Phi(L)$, we see that $\Phi(L)$ is a finite set.
\end{proof}

\section{The case of rank $2$}\label{sec:rank2}

In this section,  we study root lattices of rank $2$.
Let $L$ be a (not necessarily irreducible) root lattice over $\cO$ of rank $2$.
As we see later,  the situation is quite different from when it has a rank greater than $2$.
\textbf{For example,  $\Phi(L)$ is not necessary a finite set when $[K:\Q]=\infty$.
Moreover,  the existence of a root lattice of type $I_2(m)$ for a given integer $m$ heavily depends on the base field $K$.}

For any positive integer $n$,  let $\zeta_{n}$ denote a primitive $n$-th root of unity.
Note that $\zeta_{n} \not\in K$ for $n>2$ since $K$ is totally real. 
Let $\mu_{\infty} \subset \bar{K}^\times$ denote the group of all roots of unity in $\bar{K}$.
For any $\zeta\in\mu_\infty$,  we put $\zeta^+ \coloneqq \zeta + \zeta^{-1}$.
Note that $\mu_{\infty}\simeq\Q/\Z$ and any two isomorphic subgroups of $\mu_{\infty}$ are equal. 
For any positive integer $n$,  set $\mu_n \coloneqq \{\zeta\in\mu_\infty \mid \zeta^n=1\}\simeq\Z/n\Z$.

\subsection{Directed Graph $Q_K$}

\begin{definition}
We define a directed graph $Q_K$ associated with $K$ as follows:
\begin{itemize}
    \item The set of vertices (also denoted by $Q_K$) is given by
    \[
    Q_K \coloneqq  \{n > 1 \mid \zeta_{2n}^{+} \in K \} = \{n > 1 \mid [K(\zeta_{2n}) : K] = 2\}.
    \]
    \item For any $x, y \in Q_K$, there is a directed edge from $x$ to $y$ if $x$ divides $y$ and $y/x$ is a prime number. 
    If there is a path from $x$ to $y$, then we write $x \preceq y$.
\end{itemize}
We write $\pi_0(Q_K)$ for the set of maximal connected subgraphs (connected components) of $Q_K$.
\end{definition}

\begin{lemma}\label{lem:Q_K_properties}
\begin{itemize}
    \item[(1)] If $[K : \Q] < \infty$, then the graph $Q_K$ is finite. 
    \item[(2)] If $n \in Q_K$ and $1 \neq m \mid n$, then $m \in Q_K$. 
    \item[(3)] If $m, n \in Q_K$ and $\gcd(m,n) > 1$, then $\mathrm{lcm}(m,n) \in Q_K$. 
\end{itemize}
\end{lemma}

\begin{proof}
(1) For any positive integer $n \in Q_K$, we have 
    \[
    |(\Z/2n\Z)^\times| = [\Q(\zeta_{2n}) : \Q] \leq [K(\zeta_{2n}) : \Q] = 2 [K : \Q].
    \]
Since $[K : \Q] < \infty$ by assumption, the set of positive integers $n$ satisfying $|(\Z/2n\Z)^\times|  \leq  2 [K : \Q]$ is finite. Hence the set $Q_K$ is finite. 

(2)  When $m \mid n$, we have $K(\zeta_{2m}) \subset K(\zeta_{2n})$.  
Hence, if $n \in Q_K$, then $K(\zeta_{2m}) = K$ or $K(\zeta_{2n})$, since $[K(\zeta_{2n}) : K] = 2$.  
As $K$ is a totally real field and $m > 1$, we have $\zeta_{2m} \not\in  K$, which implies that $[K(\zeta_{2m}) : K] = [K(\zeta_{2n}) : K] = 2$, and therefore $m \in Q_K$. 

(3) For notational simplicity,  set $d \coloneqq \gcd(m,n)$,  $m' \coloneqq m/d$,  and $n' \coloneqq n/d$. 
Since $m,n \in Q_K$,  we have $\Q(\zeta_{2m}^+, \zeta_{2n}^+) \subset K$.
Thus, it suffices to show that
    \[
    \Q(\zeta_{2m}^+, \zeta_{2n}^+) = \Q(\zeta_{2dm'n'}^+).
    \]
The inclusion $\Q(\zeta_{2m}^+, \zeta_{2n}^+) \subset \Q(\zeta_{2dm'n'}^+)$ is clear. 
Therefore,  it is enough to verify that
    \[
    [\Q(\zeta_{2dm'n'}^+) : \Q(\zeta_{2d}^+)] 
    = [\Q(\zeta_{2m}^+, \zeta_{2n}^+) : \Q(\zeta_{2d}^+)].
    \]
Since $d \neq 1$, we have 
    \[
    [\Q(\zeta_{2dm'n'}^+) : \Q(\zeta_{2d}^+)] 
    = \frac{[\Q(\zeta_{2dm'n'}) : \Q]/2}{[\Q(\zeta_{2d}) : \Q]/2} 
    = [\Q(\zeta_{2dm'n'}) : \Q(\zeta_{2d})].
    \]
On the other hand the equality $\Q(\zeta_{2m}^+) \cap \Q(\zeta_{2n}^+) = \Q(\zeta_{2d}^+)$ implies
    \begin{align*}
    [\Q(\zeta_{2m}^+, \zeta_{2n}^+) : \Q(\zeta_{2d}^+)] 
    &= [\Q(\zeta_{2m}^+) : \Q(\zeta_{2d}^+)] \cdot [\Q(\zeta_{2n}^+) : \Q(\zeta_{2d}^+)] \\
    &= [\Q(\zeta_{2m}) : \Q(\zeta_{2d})] \cdot [\Q(\zeta_{2n}) : \Q(\zeta_{2d})] \\
    &= [\Q(\zeta_{2m}, \zeta_{2n}) : \Q(\zeta_{2d})] 
    = [\Q(\zeta_{2dm'n'}) : \Q(\zeta_{2d})].
    \end{align*}
Hence we obtain the desired identity.
\end{proof}

\begin{definition}
For any subset $S \subset Q_K$,  we write  $\mu(S)$ for the subgroup of $\mu_\infty$ generated by the set $\{ \zeta_{2n} \mid n \in S  \}$. 
\end{definition}

The next corollary immediately follows from \cref{lem:Q_K_properties}. 

\begin{corollary}\label{cor:Q_K_properties}
\begin{itemize}
    \item[(1)] For any distinct  $C_1, C_2 \in \pi_0(Q_K)$,  we have $\mu(C_1) \cap \mu(C_2) = \{\pm 1\}$. 
    \item[(2)] The graph $Q_K$ is recovered from the set $\{\mu(C) \mid C \in \pi_0(Q_K)\}$ as follows:
    \[
    Q_K = \left\{ n > 1 \, \middle|\,  \zeta_{2n} \in \bigcup_{C \in \pi_0(Q_K)} \mu(C) \right\}. 
    \]
\end{itemize}
\end{corollary}

\begin{remark}
Suppose that $\{H_i\}_{i\in I}$ is a family of subgroups of $\mu_\infty$ indexed by a set $I$,  satisfying $H_i \cap H_j = \{\pm 1\}$ for any distinct $i,  j\in I$.
We associate to it a totally real field 
    \[
    K \coloneqq \Q(\zeta_{2n}^+ \mid \zeta_{2n} \in H_i, \ i\in I).
    \]
Then one can see that $\{H_i\}_{i\in I}$ coincides with $\{\mu(C) \mid C \in \pi_0(Q_K) \}$. 
This means that the properties (2) and (3) of \cref{lem:Q_K_properties} characterize the family of directed graphs $\{Q_K\}_K$, where $K$ runs over all totally real fields.   
\end{remark}

\begin{lemma}
The map
    \[
    \pi_0(Q_K) \to 
    \{\text{cyclotomic quadratic extensions of } K\}; 
    \quad C \mapsto K(\mu(C)),
    \]
is a bijection.
\end{lemma}

\begin{proof}
This map is well-defined since $K(\zeta_{2m}) = K(\zeta_{2n})$ for any $m,n \in Q_K$ with $m \preceq n$. 
For any $m,n \in Q_K$, if $K(\zeta_m)=K(\zeta_n)$, then $K(\zeta_{\mathrm{lcm}(m,n)}) = K(\zeta_m) = K(\zeta_n)$.
Hence $\mathrm{lcm}(m,n) \in Q_K$ by definition. This proves the injectivity.
The surjectivity is immediate from the definition of $Q_K$.
\end{proof}

\subsection{Root lattice of rank $2$}

Let $\zeta \in \mu_\infty$ be a root of unity satisfying $[K(\zeta) : K] = 2$ and $\sigma$ the non-trivial element of $\Gal(K(\zeta)/K)$. 
We define the symmetric $K$-bilinear form $\langle \cdot,  \cdot\rangle \colon K(\zeta) \times K(\zeta)  \to K$ by 
    \[
    \langle x,   y\rangle =  x\sigma(y) + \sigma(x)y,  \qquad x,  y\in K(\zeta).
    \]
In the next proposition,  we construct a root lattice over $\cO$ of rank $2$ as an order in $K(\zeta)$.

\begin{proposition}\label{prop:root_lattce_rank2_construction}
\begin{itemize}
\item[(1)] The bilinear form $\langle \cdot, \cdot\rangle$ is totally positive definite.    
\item[(2)] The order $\cO[\zeta] \subset K(\zeta)$ is a root lattice over $\cO$ of rank $2$. 
\item[(3)] We have $\Phi(\cO[\zeta]) = \mu_\infty\cap\cO[\zeta]$. 
\end{itemize}
\end{proposition}
\begin{proof}
(1) Let $\iota \colon K(\zeta) \hookrightarrow \C$ be an embedding.
Since $K(\zeta)/K$ is a CM-extension,  we have $\iota(\sigma(x)) = \overline{\iota(x)}$ for any $x\in K(\zeta)$,  where $\bar{\cdot}$ denotes the complex conjugation. 
Hence $\iota(\langle x, x\rangle) = 2|\iota(x)|^2$. 
This proves that $\langle \cdot, \cdot\rangle$ is totally positive definite.

(2) Since $[K(\zeta) : K] = 2$, the order $\cO[\zeta]$ is a free $\cO$-module of rank $2$ with basis $\{1, \zeta\}$. 
Since $\langle 1, 1\rangle = \langle \zeta, \zeta\rangle = 2$,  the order $\cO[\zeta]$ is a root lattice over $\cO$ of rank $2$.

(3) For any root $x \in \Phi(\cO[\zeta])$ we have $x\sigma(x) = 1$,  in other words
    \[
    \Phi(\cO[\zeta]) \subset (\cO[\zeta]^\times)^{\sigma = -1},
    \]
where $(\cdot)^{\sigma=-1}$ denotes the subset of $x$ satisfying $\sigma(x)=x^{-1}$.
By Dirichlet's unit theorem, we have
    \[
    (\cO_{K(\zeta)}^\times \otimes_{\Z} \Q)^{\sigma = -1}
    = (\cO^\times \otimes_{\Z} \Q)^{\sigma = -1}
    = \{0\}.
    \]
This implies that $(\cO[\zeta]^\times)^{\sigma = -1} = (\cO[\zeta]^\times)_{\mathrm{tors}}$,  where $(\cdot)_{\mathrm{tors}}$ denotes the torsion subgroup.
Hence we obtain $\Phi(\cO[\zeta])\subset\mu_\infty\cap\cO[\zeta]$.
The converse inclusion is immediate from the definition of $\langle \cdot, \cdot\rangle$.
\end{proof}

Recall that $L$ is a (not necessarily irreducible) root lattice over $\cO$ of rank $2$.
From \cref{cos1},  for any $\alpha,  \beta\in\Phi(L)$ there exists $\zeta \in \mu_{\infty}$ satisfying $\alpha\cdot\beta = \zeta^+$.

\begin{lemma}\label{lem:rank2_root_lattice_model}
The root lattice $L$ is isomorphic to $\cO[\zeta]$ for some $\zeta \in \mu_{\infty}$.
\end{lemma}

\begin{proof}
Since $L$ is generated by $\Phi(L)$, one can take an $\cO$-basis $\{\alpha, \beta\} \subset \Phi(L)$ of $L$. 
Suppose that $\zeta \in \mu_{\infty}$ satisfies $\alpha\cdot\beta = \zeta^+$.
We define the $\cO$-homomorphism $\varphi \colon L \to \cO[\zeta]$ by $\varphi(\alpha)=1$ and $\varphi(\beta)= \zeta$.
This is an isomorphism of root lattices over $\cO$ by definition. 
\end{proof}

\begin{definition}
Let $\mu(L)$ denote the subgroup of $\mu_\infty$ generated by the set
    \[
    \{\zeta \in \mu_{\infty} \mid 
    \zeta^+ = \alpha\cdot\beta ,  \quad \alpha,  \beta\in\Phi(L)
    \}.  
    \]
\end{definition}

Note that for any $\zeta\in\mu_\infty$ one has
    \[
    \mu(\cO[\zeta])=\mu_\infty\cap\cO[\zeta]=\Phi(\cO[\zeta]),
    \]
where the latter equality is \cref{prop:root_lattce_rank2_construction} (3). 
The next lemma shows that the group $\mu(L)$ is a complete invariant of the isomorphism classes of root lattices of rank $2$.

\begin{lemma}\label{lem:mu(L)_is_invarinant_under_isom}
Let $L_1$ and $L_2$ be root lattices of rank $2$ over $\cO$. 
Then $L_1 \simeq L_2$ as root lattices if and only if $\mu(L_1) \simeq \mu(L_2)$ as groups,  or equivalently $\mu(L_1) = \mu(L_2)$.
\end{lemma}

\begin{proof}
The only if part is obvious.
We prove the converse direction.
Suppose that $\mu(L_1) \simeq \mu(L_2)$.
By \cref{lem:rank2_root_lattice_model},  one may assume that
$L_1 = \cO[\zeta]$ and $L_2 = \cO[\zeta']$,  with some $\zeta, \zeta' \in \mu_\infty$.
Then \cref{prop:root_lattce_rank2_construction} combined with $\mu(L_1) = \mu(L_2)$ yields $\mu_\infty\cap\cO[\zeta] = \mu_\infty\cap\cO[\zeta']$.
This implies that $\zeta' \in \cO[\zeta]$ and $\zeta \in \cO[\zeta']$,  and hence
$\cO[\zeta] = \cO[\zeta']$. 
In particular we obtain $L_1\simeq L_2$.
\end{proof}

\begin{definition}
Let $H$ be a subgroup of $\mu_\infty$ containing $-1$.
We say that $L$ is \emph{of type $H$} if $\mu(L)$ is isomorphic to $H$,  or equivalently $\mu(L) = H$. 
\end{definition}

\begin{remark}
Suppose that $|H|=2m<\infty$ and $L$ is of type $H$.
Note that  $H = \mu_{2m} \simeq \Z/2m\Z$.
It follows from \cref{prop:root_lattce_rank2_construction} (3) and \cref{lem:rank2_root_lattice_model} that $|\Phi(L)|=|H|<\infty$.
In this case,  $L$ is of type $I_2(m)$ when $m>2$ and of type $A_1\times A_1$ when $m=2$.
In particular,  a root lattice of type $H$ is NOT irreducible if and only if $|H|=4$.
\end{remark}

\subsection{Preliminaries on cyclotomic fields}
From the results in the previous subsection,  it suffices to study roots of unity in the order $\cO[\zeta]$ for $\zeta\in\mu_\infty$.
Before that,  we prepare some useful lemmas on cyclotomic fields. 

\begin{lemma}\label{lem:ramification_cyclotomic}
Let $n$ be a positive integer.
\begin{itemize}
\item[(1)] Suppose that $n$ is not a prime power.
Then $\Q(\zeta_{2n})/\Q(\zeta_{2n}^+)$ is unramified at all finite places. 
\item[(2)] Suppose that $n=p^k$ with a prime number $p$ and a positive integer $k$.
Then $\Q(\zeta_{2n})/\Q(\zeta_{2n}^+)$ is ramified at all places above $p$ and unramified at other places. 
\end{itemize}
\end{lemma}
\begin{proof}
This is \cite[Proposition 2.15]{Washington97}. 
Note that the statement therein requires a minor modification; it becomes valid after replacing $n$ with $2n$.
\end{proof}

\begin{lemma}\label{lem:ramification_real_cyclotomic}
Let $p$ and $q$ be prime numbers,  not necessarily distinct. 
Let $e$ be a positive integer. 
Then $\Q(\zeta_{2p^eq}^+)/\Q(\zeta_{2p^e}^+)$ is ramified at all places above $p$. 
\end{lemma}

\begin{proof}
The case $p = q$ is clear, since  $\Q(\zeta_{2p^e q}^+)/\Q$ is totally ramified at $p$ in this case. 
We assume that $p \neq q$.
It follows from \cref{lem:ramification_cyclotomic} that $\Q(\zeta_{2p^eq})/\Q(\zeta_{2p^eq}^+)$ is unramified at all finite places. 
For any number field $F$,  let $e_p(F)$ denote the ramification index of the extension $F/\Q$ at $p$.
Since $e_p(\Q(\zeta_{2p^e q}))=(p-1)p^{e-1}$,  we obtain $e_p(\Q(\zeta_{2p^e q}^+)) = (p-1)p^{e-1}$.
On the other hand, we have $e_p(\Q(\zeta_{2p^e}^+)) = [\Q(\zeta_{2p^e}^+) : \Q] = (p-1)p^{e-1}/2$. 
Since $e_p(\Q(\zeta_{2p^e}^+)) < e_p(\Q(\zeta_{2p^e q}^+))$, we conclude that $\Q(\zeta_{2p^eq}^+)/\Q(\zeta_{2p^e}^+)$ is ramified at all places above $p$. 
\end{proof} 

Take $m \in Q_K$. 
Note that $K \cap \Q(\zeta_{2m}) =  \Q(\zeta_{2m}^+)$. 

\begin{lemma}\label{lem:field_isom}
The canonical homomorphism $K \otimes_{\Q(\zeta_{2m}^+)} \Q(\zeta_{2m}) \stackrel{\sim}{\to} K(\zeta_{2m})$ is an isomorphism. 
In other words,  $K$ and $\Q(\zeta_{2m})$ are linearly disjoint over $\Q(\zeta_{2m}^+)$.
\end{lemma}

\begin{proof}
Since $\Q(\zeta_{2m})/\Q(\zeta_{2m}^+) $ is a Galois extension and $K \cap \Q(\zeta_{2m}) =  \Q(\zeta_{2m}^+)$, the homomorphism $K \otimes_{\Q(\zeta_{2m}^+)} \Q(\zeta_{2m}) \longrightarrow K(\zeta_{2m})$ is injective. 
The surjectivity is clear from the definition of $K(\zeta_{2m})$. 
\end{proof}

\begin{lemma}\label{lem:faithfully-flat}
Let $M/L$ be an extension of number fields with $[L : \Q] < \infty$.
The inclusion $\cO_L \hookrightarrow \cO_M$ is faithfully flat. 
\end{lemma}

\begin{proof}
Take a tower of subfields $M_1\subset M_2 \subset \cdots \subset M$ satisfying 
    \[
    M = \bigcup_{i=1}^{\infty} M_i  \quad \text{ and } [M_i : \Q] < \infty \text{ for all $i$}
    \] 
so that we have $\cO_M = \varinjlim\cO_{M_i}$.
By  \cite[\href{https://stacks.math.columbia.edu/tag/090N}{Tag 090N}]{stacks-project}  we are reduced to show each $\cO_{M_i}$ is faithfully flat over $\cO_L$.
Now we may assume that $[M : \Q] < \infty$.  

Since $\cO_L$ is a Dedekind domain and $\cO_M$ is a torsion-free $\cO_L$-module,  the $\cO_L$-module $\cO_M$ is flat. On the other hand the map $\mathrm{Spec}(\cO_M) \to \mathrm{Spec}(\cO_L)$ is surjective,  and hence the inclusion $\cO_L \hookrightarrow  \cO_M$ is faithfully flat.  
\end{proof}

\begin{corollary}\label{cor:injective_ring_hom}
The canonical homomorphism $\cO \otimes_{\Z[\zeta_{2m}^+]} \Z[\zeta_{2m}] \to \cO_{K(\zeta_{2m})}$ is injective. 
\end{corollary}

\begin{proof}
Since we have isomorphisms $\cO \otimes_{\Z[\zeta_{2m}^+]} \Z[\zeta_{2m}] \otimes_{\Z} \Q \stackrel{\sim}{\longrightarrow} K \otimes_{\Q(\zeta_{2m}^+) } \Q(\zeta_{2m}) \stackrel{\sim}{\longrightarrow} K(\zeta_{2m})$ by \cref{lem:field_isom}, the kernel of  the  homomorphism $\cO \otimes_{\Z[\zeta_{2m}^+]} \Z[\zeta_{2m}] \to \cO_{K(\zeta_{2m})}$ is a torsion $\Z$-module. 
However, since $\cO \otimes_{\Z[\zeta_{2m}^+]} \Z[\zeta_{2m}]$ is torsion-free by \cref{lem:faithfully-flat}, we conclude that the kernel of this homomorphism is trivial. 
\end{proof}

\subsection{Main Result for Rank 2 Root Lattices} 
We define the subsets $\cP_K$  and $\cR_K$ of the power set of $Q_K$ as follows:
    \begin{align*}
    \cP_K &=  \{ \{q\}  \mid \text{$q \in Q_K$ is a prime power} \},  \\
    \cR_K &=  \{ C \in \pi_0(Q_K) \mid \text{$C$ contains an integer that is not a prime power} \}. 
    \end{align*}

\begin{proposition}\label{prop:torsion_composite}
Let $C \in \cR_K$ and suppose that $m, m' \in C$ are not prime powers.
Then we have $\cO[\zeta_{2m}] = \cO[\zeta_{2m'}]$. 
In particular, $\mu(\cO[\zeta_{2m}]) = \mu(C)$. 
\end{proposition}
\begin{proof}
We may assume that $m \preceq m'$ by \cref{lem:Q_K_properties} (3). 
Since $\Z[\zeta_{2m'}^+] \subset \cO$ by the definition of $Q_K$, it suffices to show $\Z[\zeta_{2m'}^+, \zeta_{2m}] = \Z[\zeta_{2m'}]$.
Since $m$ is not a prime power by assumption,  \cref{lem:ramification_cyclotomic} shows that the extension $\Q(\zeta_{2m})/\Q(\zeta_{2m}^+)$ is unramified at all finite places.
Hence the discriminants of the extensions
    \[
    \Q(\zeta_{2m})/\Q(\zeta_{2m}^+)
    \quad \text{and} \quad
    \Q(\zeta_{2m'}^+)/\Q(\zeta_{2m}^+)
    \]
are coprime. This fact implies that $\Z[\zeta_{2m'}^+, \zeta_{2m}] = \Z[\zeta_{2m'}]$ (see, for example, \cite[Proposition 2.11]{Neukirch13}). 
\end{proof}

\begin{proposition}\label{prop:torsion_prime_power}
If $\{q\} \in \cP_K$, then we have $\mu(\cO[\zeta_{2q}]) = \mu_{2q}$. 
\end{proposition}
\begin{proof}
Let $\ell$ be a  prime number. It suffices to show that $\cO[\zeta_{2q}] \neq \cO[\zeta_{2\ell q}]$.
By \cref{cor:injective_ring_hom},  this is equivalent to saying that the natural ring homomorphism
    \[
    \cO \otimes_{\Z[\zeta_{2q}^+]} \Z[\zeta_{2q}]
    \to \cO \otimes_{\Z[\zeta_{2\ell q}^+]} \Z[\zeta_{2\ell q}]    
    \]
is not an isomorphism.
Since $\cO$ is a faithfully flat $\Z[\zeta_{2\ell q}^+]$-algebra by \cref{lem:faithfully-flat},  this is equivalent to that the ring homomorphism 
    \begin{equation}\label{eq:hom2}
    \Z[\zeta_{2\ell q}^+] \otimes_{\Z[\zeta_{2q}^+]} \Z[\zeta_{2q}]
    \to \Z[\zeta_{2\ell q}]
    \end{equation}
is not an isomorphism.

Let $p$ be the prime number dividing $q$. 
Then the ring $\Z[\zeta_{2\ell q}] \otimes_{\Z} \Z_p$ 
is isomorphic to a direct product of the complete discrete valuation ring $\Z_p[\zeta_{2\ell q}]$. 
In particular $\Z[\zeta_{2\ell q}] \otimes_{\Z} \Z_p$ is regular. 
On the other hand,  the ring $\Z[\zeta_{2\ell q}^+] \otimes_{\Z[\zeta_{2q}^+]} \Z[\zeta_{2q}] \otimes_{\Z} \Z_p$ is isomorphic to a direct product of the semi-local ring
    \begin{equation}\label{semi-local}
    \Z_p[\zeta_{2\ell q}^+] \otimes_{\Z_p[\zeta_{2q}^+]} \Z_p[\zeta_{2q}].
    \end{equation}
Since both extensions $\Q_p(\zeta_{2q})/\Q_p(\zeta_{2q}^+)$ and $\Q_p(\zeta_{2\ell q}^+)/\Q_p(\zeta_{2q}^+)$ are ramified by \cref{lem:ramification_cyclotomic} and \cref{lem:ramification_real_cyclotomic}, the semi-local ring \eqref{semi-local} is not regular. 
Therefore the homomorphism \eqref{eq:hom2} is not an isomorphism. 
\end{proof}

Now we can prove the following classification theorems of rank $2$ root lattices.

\begin{theorem}\label{thm:rank_2_root_lattice}
Suppose that $L$ is a root lattice over $\cO$ of rank $2$.
\begin{itemize}
\item[(1)] We have $\mu(L)= \mu(C)$ for some $C \in \cP_K \cup \cR_K$. 
\item[(2)]  For each $C \in  \cP_K \cup \cR_K$,  there exists a root lattice of rank $2$ over $\cO$ of type $\mu(C)$. 
\end{itemize}
In other words, we have the following bijection:
    \[
    \{ \text{root lattices over $\cO$ of rank $2$} \}_{/\simeq}
    \stackrel{\sim}{\to} \{\mu(C) \mid C\in\cP_K\cup\cR_K\};  \quad L \mapsto \mu(L).
    \]
Note that since $\cP_K\cup\cR_K  \stackrel{\sim}{\to} \{\mu(C) \mid C\in\cP_K\cup\cR_K\}; C \mapsto \mu(C)$ is a bijection, root lattices over $\mathcal{O}$ of rank $2$ can be classified by $\mathcal{P}_K \cup \mathcal{R}_K$.
\end{theorem}

\begin{proof}
(1) By \cref{lem:rank2_root_lattice_model} and \cref{lem:mu(L)_is_invarinant_under_isom},  one may assume that
$L = \cO[\zeta_{2n}]$ with some $\zeta_{2n} \in \mu_\infty$.
If $n$ is not a prime power,  let $C\in\pi_0(Q_K)$ be the connected component containing $n$. 
Then $C$ belongs to $\cR_K$ and the equality $\mu(L) = \mu(C)$ follows from \cref{prop:torsion_composite}.
If $n$ is a prime power,  then $\{n\} \in \cP_{K}$ and the equalities $\mu(L) = \mu_{2n}=\mu(\{n\})$ follows from \cref{prop:torsion_prime_power}. 

(2) This follows from \cref{prop:torsion_composite} and \cref{prop:torsion_prime_power}. 

The last statement follows from (1), (2), and \cref{lem:mu(L)_is_invarinant_under_isom}.
\end{proof}

\begin{theorem}\label{thm:rank_2_root_lattice2}
Let $\widetilde{K}$ be an totally real extension of $K$ with ring of integers $\widetilde{\cO}$. 
Suppose that $L$ is a root lattice over $\cO$ of rank $2$.
Set $\widetilde{L} \coloneqq L\otimes_\cO\widetilde{\cO}$,  a root lattice over $\widetilde{\cO}$.
\begin{itemize} 
\item[(1)] Let $\{q\} \in \cP_K$ and suppose that $L$ is of type $\mu_{2q}$.
Then $\Phi(\widetilde{L})=\Phi(L)$.
In particular $\widetilde{L}$ is of type $\mu_{2q}$. 
\item[(2)] Let $C \in \cR_K$ and suppose that $L$ is of type $\mu(C)$.
Let $C'$ denote the connected component of $Q_{\widetilde{K}}$ containing $C$.
Then $\widetilde{L}$ is of type $\mu(C')$.  
\end{itemize}
\end{theorem}

\begin{proof}
Let $n \in Q_K$. 
Then $\cO[\zeta_{2n}]$ (resp.\,$\widetilde{\cO}[\zeta_{2n}]$) is a free $\cO$-module (resp.\,$\widetilde{\cO}$-module) of rank $2$ with a basis $\{1, \zeta_{2n}\}$. 
Hence,  the canonical homomorphism $\cO[\zeta_{2n}] \otimes_{\cO} \widetilde{\cO} \xrightarrow{\sim} \widetilde{\cO}[\zeta_{2n}]$ is an isomorphism.
Thus the assertion follows from \cref{prop:torsion_composite} and \cref{prop:torsion_prime_power}.
\end{proof}

\subsection{Examples}

\subsubsection{} 

Let $K \coloneqq \Q(\zeta_{28}^+,  \zeta_{30}^+)$. 
The directed graph $Q_K$ is as follows:
    
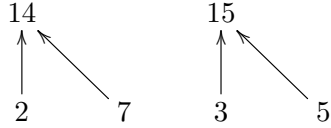
\[
    \xymatrix{
    14   &  \\
    2 \ar[u] & 7  \ar[ul] 
    } \qquad 
    \xymatrix{
    15   &  \\
    3 \ar[u] & 5  \ar[ul] 
    } 
    \]
Hence we have 
    \[
    \cP_K = \{\{2\}, \{3\}, \{5\},  \{7\}\}, \quad 
    \cR_K = \pi_0(Q_K) = \{\{ 2,7,14\}, \{3,5,15\}\}. 
    \]
Thus there are five isomorphism classes of irreducible root lattices of rank $2$ over $\cO$,  namely of type $I_2(3)$, $I_2(5)$,  $I_2(7)$, $I_2(14)$,  and $I_2(15)$. 

Let $\widetilde{K} \coloneqq  \Q(\zeta_{420}^+)$. 
The directed graph $Q_{\widetilde{K}}$ is as follows:
    
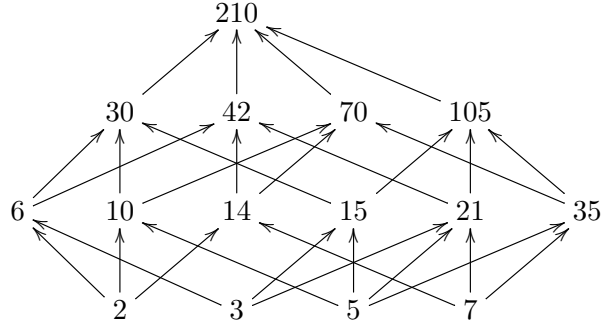
\[
    \xymatrix{
    &&210&&&  \\
    &30 \ar[ur]  & 42 \ar[u]  & 70 \ar[ul] & 105\ar[ull] &  \\
    6 \ar[ur] \ar[urr] & 10 \ar[u] \ar[urr]   & 14 \ar[u] \ar[ur] 
    & 15 \ar[ull] \ar[ur] & 21 \ar[ull] \ar[u]  & 35 \ar[ull] \ar[ul]  \\
    & 2 \ar[ul] \ar[u] \ar[ur]  & 3 \ar[ull] \ar[ur] \ar[urr] 
    & 5 \ar[ull] \ar[u] \ar[urr] \ar[ur] & 7 \ar[ull] \ar[u] \ar[ur] &
    }
    \]
Thus there are four isomorphism classes of irreducible root lattices of rank $2$ over $\widetilde{\cO}$,  namely of type $I_2(3)$, $I_2(5)$,  $I_2(7)$,  and $I_2(210)$.
Note that we have 
    \[
    \widetilde{\cO}[\zeta_{28}] \cong \cO[\zeta_{420}] 
    \cong \widetilde{\cO}[\zeta_{30}]. 
    \]
In this way, in rank $2$, root lattices of different types can sometimes become isomorphic after scalar extension. 

\subsubsection{}\label{sec:example_rank2_ab}
Let $K \coloneqq  \Q(\zeta_{n}^+ \mid n \in \Z_{>0} )$ be the maximal totally real subfield of the maximal abelian extension $\Q^{\mathrm{ab}}$ of $\Q$. 
Then we see that $Q_K = \Z_{>1}$,  which is connected and 
    \[
    \cP_K = \{ \{p^n\} \mid p\colon\text{a prime number}, \  n\in\Z_{>0} \},  \quad
    \cR_K = \pi_0(Q_K) = \{Q_K\}. 
    \]
Therefore for any prime power $q$, there is  a root lattice of rank $2$ over $\cO$ of type $\mu_{2q}$. 
Also there exists a root lattice of type $\mu_{\infty}$. 
More concretely, it is given by the ring of integers $\cO_{\Q^{\mathrm{ab}}}$ of $\Q^{\mathrm{ab}}$ since $\cO_{\Q^{\mathrm{ab}}} = \cO[\zeta_{12}]$ by \cref{prop:torsion_composite}. 
Since $\cO \otimes_{\mathbb{Z}[\zeta_{12}^+]} \mathbb{Z}[\zeta_{12}]
= \cO_{\Q^{\mathrm{ab}}}$,
the root lattice of type $\mu_{\infty}$ given by $\cO_{\Q^{\mathrm{ab}}}$
can be obtained as a scalar extension of the root lattice $\Z[\zeta_{12}]$ of type $\mu_{12}$
over $\Z[\zeta_{12}^+]$.

\subsubsection{} 
Suppose that $K \coloneqq \bigcup_{n>0}\Q(\zeta_{p^n}^+)$,  where $p$ is a prime number. 
Then we have $Q_K = \{ p^n \mid n\in\Z_{>0} \}$ and 
    \[
    \cP_K = \{ \{p^n\} \mid n\in\Z_{>0} \},  \quad
    \cR_K = \emptyset. 
    \]
For any positive integer $n$,  there exists a root lattice of rank $2$ over $\cO$ of type $\mu_{2p^n}$. 

On the other hand,  a root lattice of type $\mu_{2p^\infty}\coloneqq\bigcup_{n>0}\mu_{2p^n}$ does not exist.
This is because the ring of integers $\cO_{K(\zeta_p)}$ is not a root lattice of type $\mu_{2p^\infty}$,  unlike the situation in  \cref{sec:example_rank2_ab},  since it 
is not finitely generated over $\cO$.

\bibliography{references}

\end{document}